\documentclass[11pt,centertags,oneside, reqno]{amsart}
\usepackage[T1]{fontenc}
\usepackage[english]{babel}
\usepackage{amsmath, mathtools, amstext,amsthm,a4,amssymb,amscd}
\usepackage[nobysame]{amsrefs}
\usepackage[mathscr]{eucal}
\usepackage{mathrsfs}
\usepackage{epsf}
\usepackage{typearea}
\usepackage{charter}
\usepackage[hidelinks]{hyperref}
\usepackage[shortlabels]{enumitem}
\setlist[enumerate]{nosep}
\usepackage{float}
\usepackage{caption}
\usepackage{tikz}

\usetikzlibrary{decorations.text}
\usetikzlibrary{decorations.pathreplacing}
\usetikzlibrary{patterns}
\usetikzlibrary{cd}
\usetikzlibrary{matrix}

\usepackage[a4paper,width=16.2cm,top=2.5cm,bottom=2.5cm]{geometry}

\numberwithin{equation}{section}

\title{Asymptotics for Toeplitz operators with symbol an indicator function}
\author{Razvan Apredoaei}

\address{Razvan Apredoaei \\ Université Paris Cité, Sorbonne Université, CNRS, IMJ-PRG, F-75013 Paris, France.} 
\email{razvan.apredoaei@imj-prg.fr}

\date{}

\newtheorem{thm}{Theorem}[section]
\newtheorem*{thm*}{Theorem}
\newtheorem{lem}[thm]{Lemma}
\newtheorem{prop}[thm]{Proposition}
\newtheorem{cor}[thm]{Corollary}
\newtheorem{rem}[thm]{Remark}
\newtheorem{rems}[thm]{Remarks}

\newcommand{\scal}[2]{\left\langle{#1},{#2}\right\rangle}
\newcommand{\bbone}{\mathbf{1}}
\newcommand{\field}[1]{\mathbb{#1}}

\newcommand{\R}{\field{R}}
\newcommand{\C}{\field{C}}
\newcommand{\N}{\field{N}}

\newcommand{\be}{\begin{eqnarray}}
\newcommand{\ee}{\end{eqnarray}}

\newcommand{\vare}{\varepsilon}
\newcommand{\varp}{\varphi}

\def\cC{\mathscr{C}}

\def\mF{\mathcal{F}}
\def\mH{\mathcal{H}}

\def\mJ{\mathcal{J}}

\def\mL{\mathcal{L}}

\def\mO{\mathcal{O}}
\def\mP{\mathcal{P}}
\def\mQ{\mathcal{Q}}

\def\mU{\mathcal{U}}

\def\kn{\mathfrak{n}}

\def\Re{{\rm Re}}
\def\Im{{\rm Im}}

\DeclareMathOperator{\End}{End}

\DeclareMathOperator{\Ker}{Ker}

\DeclareMathOperator{\rank}{rk}
\DeclareMathOperator{\Id}{Id}

\DeclareMathOperator{\tr}{Tr}

\DeclareMathOperator{\vol}{vol}

\DeclareMathOperator{\inj}{inj}
\DeclareMathOperator{\spec}{Spec}

\DeclareMathOperator{\Ric}{Ric}

\begin{document}

\begin{abstract}
We prove an off-diagonal expansion of the kernel of the Toeplitz operator whose symbol is the indicator function of a compact domain with smooth boundary in a complete symplectic manifold of bounded geometry. Using our approach, we extend two results to the non-compact setting: the first concerns the asymptotics of the trace of polynomials in this operator, and the second establishes a Weyl law for this Toeplitz operator. 
\end{abstract}

\maketitle
\tableofcontents

\section{Introduction}

Let $(X, \omega)$ be a compact Kähler manifold and $L$ a holomorphic Hermitian line bundle over $X$ such that $L$ is prequantum, i.e., the curvature of the Chern connexion of $L$ coincides with the symplectic form $\omega$ (see \eqref{prequantum}). Let  $P_p$ be the projection on the space $\mH_p$ of holomorphic sections of $L^p$. The projector $P_p$ has a smooth kernel known as the \textit{Bergman kernel}. Let $f$ be a bounded function on $X$. The Toeplitz operator with symbol $f$ is the family of operators defined by $T_{f, p} := P_pfP_p$.  In the particular case where $f = \bbone_W$ is the indicator function of a compact domain $W$ with smooth boundary $V = \partial W$, the operator $T_{\bbone_W, p}$ behaves semiclassically as an almost-projection. Its spectrum concentrates near $0$ and $1$ with order $p^n$ and the number of eigenvalues contained in $[a, b]\subset ]0, 1[$ is given by the following Weyl law:
\begin{equation}\label{Weyl}
		\Big|\spec(T_{W, p})  \cap [a, b] \Big| = p^{n-1/2} \frac{\mathrm{vol}(V)}{\sqrt{2\pi}} (\mathrm{erfc}^{-1}(a)-\mathrm{erfc}^{-1}(b)) + o(p^{n-1/2}),
\end{equation}
where $\mathrm{erfc}$ denotes the complementary error function (see \eqref{erfunction}). This behavior has been uncovered by Charles\textendash Estienne in~\cite[Theorem 1.1]{CE20} and comes from the full asymptotic expansion in $p \rightarrow \infty$ of the trace of $g(T_{\bbone_W, p})$ for $g$ polynomial, which they prove in~\cite[Theorem 1.2]{CE20}. They find an application of their results in the study of entanglement entropy for free Fermions~\cite[\S 7]{CE20}. As another application, Polterovich remarked in~\cite[\S 6]{P07} that Toeplitz operators with piecewise constant symbol may be a relevant tool to infer the homotopy type of a quantum phase space.

In this paper we are interested in the more general setting where $(X, \omega)$ is a symplectic manifold of bounded geometry endowed with a prequantum Hermitian line bundle $L$ and a Hermitian vector bundle $E$ over $X$. In this setting, the operators $T_{f, p}= P_p f P_p$ act on sections of $L^p\otimes E$, and $P_p$ denotes the projection on a closed subspace $\mH_p$ of the space of square-integrable sections $L^2(X, L^p \otimes E)$ for the metrics on $L$ and $E$. In our case $\mH_p$ is a spectral space associated to a renormalisation of the Bochner Laplacian which was also used in \cite{IOOSAdv}.

We prove asymptotics for the Toeplitz operator with symbol an indicator function of a compact domain $W$ inside $X$ with smooth boundary $V$. The novelty of this paper is fourfold:
\begin{enumerate}
\item In Theorem~\ref{ThondiagonalTwexpansion}, we obtain full on-diagonal kernel asymptotics for the Toeplitz operator $T_{W, p}$ and show that the transition across the boundary is governed by a universal error-function profile;

\item In Theorem~\ref{thmNoyauCompo}, we  get on-diagonal kernel asymptotics for $g(T_{W, p})$ with $g$ a polynomial;
\item Using (2), we extend the trace asymptotics of Charles--Estienne \cite{CE20} to complete symplectic manifolds of bounded geometry as stated in Corollaries \ref{ThCEGeneralCase} and \ref{CorTraceToeplitz};
\item We also extend the Weyl law~\eqref{Weyl} for $T_{W, p}$ to the non-compact setting (see Corollary~\ref{CorTraceToeplitz}).
\end{enumerate}

The case of an indicator function of a compact domain with smooth boundary is arguably the simplest example of a Toeplitz operator with singular symbol. The asymptotics we derive emphasize that its discontinuity creates a boundary layer phenomenon that is invisible for smooth symbols and leads to half-integer powers in the semiclassical expansions.

Two key results are essential to our proofs. The first one is the asymptotic expansion of the Bergman kernel in a neighborhood of the diagonal established by Dai--Liu--Ma \cite[Theorem 4.18]{DLM06}, Ma--Marinescu \cite[Chapter 8]{MM07} and Kordyukov--Ma--Marinescu \cite[Theorem 4.3]{KMM19}. The second one is the exponential decay property of the Bergman kernel outside the diagonal established by Ma--Marinescu~\cite[Theorem 1]{MM15}, and Kordyukov--Ma--Marinescu~\cite[Theorem 1.2]{KMM19}. 

Let us mention that previous known results on kernel asymptotics for the operators $T_{f, p}$ were obtained with stronger regularity assumptions on $f$. On a symplectic manifold and in the case where $\mH_p$ is the kernel of the $\mathrm{spin}^c$-Dirac operator, Ma\textendash Marinescu~\cite[Lemma 7.2.4, \S 8.1.2]{MM07}, and~\cite[\S 4]{MM08_TSym} proved a full off-diagonal expansion for the kernel of Toeplitz operators with symbol $f$ when $f$ is smooth. In~\cite[\S 3]{BMMP14}, Barron, Ma, Marinescu and Pinssonnault prove kernel asymptotics for Toeplitz operators when $f$ is $\cC^k$ and $\mH_p$ the kernel of the $\mathrm{spin}^c$-Dirac operator. In \cite{CharlesPolterovich}, Charles and Polterovich also studied the Berezin--Toeplitz quantization of $\cC^k$ functions. 

We now state our results. Let $(X, \omega)$ be a symplectic manifold of dimension $2n$. Let $(L, h^L)$ be a Hermitian line bundle over $X$ and $(E, h^E)$ be a Hermitian vector bundle over $X$, both endowed with Hermitian connections $\nabla^L$ and $\nabla^E$, with curvatures $R^L := (\nabla^L)^2$ and $R^E := (\nabla^E)^2$. We assume that $L$ satisfies the prequantization condition 

	\begin{equation}\label{prequantum}
		\frac{\sqrt{-1}}{2\pi}R^L = \omega.
	\end{equation}
We endow $(X, \omega)$ with a \textit{complete Riemannian metric} $g^{TX}$. Let $J : TX \rightarrow TX$ be an almost-complex structure on X compatible with $\omega$ and $g^{TX}$. Let $J_0 : TX \rightarrow TX$ be the skew-adjoint linear map which satisfies

	\begin{equation}\label{J0}
		\omega(u,v) = g^{TX}(J_0 u, v), \quad u, v \in TX,
	\end{equation}
then $J = J_0(-J_0^2)^{-1/2}$. We suppose that $(X, L, E)$ has \emph{bounded geometry}: the injectivity radius $\inj^X$ of $(X, g^{TX})$ is positive and the derivatives of any order of $R^L, R^E, J, g^{TX}$ are uniformly bounded on $X$. Let $\Delta^{L^p \otimes E}= (\nabla^{L^p \otimes E})^* \nabla^{L^p \otimes E}$ denote the Bochner Laplacian  (see~\eqref{BochnerLapl}) associated to $\nabla^L$, $\nabla^E$ and $g^{TX}$, with $(\nabla^{L^p \otimes E})^*$ the formal adjoint of $\nabla^{L^p \otimes E}$. Define by

	\begin{equation}\label{RenormBochnerL}
		\Delta_p := \Delta^{L^p \otimes E} - p\tau, \quad \mathrm{with} \quad \tau(x) := -\pi \tr \left[(JJ_0)(x)\right], \quad x \in X,
	\end{equation}
the \emph{renormalized Bochner Laplacian} acting on $\cC^{\infty}(X, L^p \otimes E)$, as originally introduced by \linebreak Guillemin and Uribe in~\cite{GU88}. Assume
	\begin{equation}\label{CourbureUnifPositive}
		\mu_0 := \underset{x \in X, u \in T_xX \setminus \{0\}}{\inf} {\frac{\sqrt{-1} R_x^L(u, Ju)}{|u|_{g_x^{TX}}^2}} > 0.
	\end{equation}

By~\cite[Theorem 1.1]{KMM19},~\cite[Corollary 1.2]{MM02}, we have the following spectral gap property: there exists $C_L > 0$ such that for any $p \in \N$ the spectrum of $\Delta_p$ satisfies

	\begin{equation}
		\spec(\Delta_p) \subset [-C_L, C_L] \cup [2p\mu_0 - C_L, + \infty[.
		\label{rBochnerspectralgap}
	\end{equation}
For $p$ large enough, as in~\cite{MM08},~\cite[Chapter 8]{MM07},~\cite{KMM19}, let $\mH_p([-C_L, C_L]) \subset L^2(X, L^p \otimes E)$ denote the closed subspace of $L^2(X, L^p \otimes E)$ given by the spectral space corresponding to $[-C_L, C_L]$.
Let $P_p$ be the spectral projector corresponding to $[-C_L, C_L]$, called the \emph{generalized Bergman projector} associated with $\Delta_p$ in the sense of~\cite{MM08}:

	\begin{equation}\label{Bergmanprojector}
		P_p : L^2(X, L^p \otimes E) \rightarrow \mH_p([-C_L, C_L]).
	\end{equation} 
By the Schwartz kernel theorem, $P_p$  has a distribution kernel with respect to the Riemannian volume form $dv_X$ of $(X, g^{TX})$. Its kernel $P_p( \cdot, \cdot) \in \cC^{\infty}(X \times X, (L^p \otimes E) \boxtimes {(L^p \otimes E)}^*)$ is called the \emph{generalized Bergman kernel} as in~\cite{MM08}. As in \cite[Remark 1.4.3, Problem 1.5]{MM07}, the smoothness of $P_p$ comes from the ellipticity of the operator $\Delta_p$. Let $d(x, x')$ denote the geodesic distance between two points $x, x' \in X$. By~\cite[Theorem 1.2]{KMM19}, ~\cite[Theorem 1]{MM15}, we have the following exponential estimates: there exists $c > 0$ such that for any $l > 0$, there exists $C_l > 0$ such that for any $p \in \N$ and $x, x' \in X$, we have 
	\begin{equation}\label{expestimates}
		|P_p(x, x')|_{\cC^l(X \times X)} \leq C_l p^{n+ \frac{l}{2}}e^{-c\sqrt{p}d(x, x')}.
	\end{equation}
 In the above inequality the norm $\cC^l(X\times X)$ denotes the pointwise $\cC^l$-seminorm of the section $P_p$ at the point $(x, x') \in X \times X$, which is the sum of the norms induced by $h^L, h^E, g^{TX}$ of the derivatives up to order $l$ of $P_p$ for $\nabla^{L^p \otimes E}$ and $\nabla^{TX}$ evaluated at $(x, x')$. The generalized Bergman kernel $P_p$ also admits an asymptotic expansion in $p$, which has been established by Ma--Marinescu~\cite[Chapter 8]{MM07} and Lu--Ma--Marinescu~\cite{LMM17}. In particular by~\cite[Problem 6.1]{MM07}, as in~\cite[Theorem 6]{MM15}, there exist smooth coefficients $b_r \in \cC^{\infty}(X, \mathrm{End}(E))$ such that for any $k,l \in \N$, there exists $C_{k, l} > 0$ such that for any $x \in X$, $p \in \N^*$,
	\begin{equation}\label{BergmanKondiag}
		\left|P_p(x, x) - \sum_{r=0}^k b_r(x)p^{n-r}\right|_{{\cC^l}(X)} \leq C_{k,l}p^{n-k-1},
	\end{equation}
moreover if we denote by $T^{(1, 0)}X := \Ker(J-\sqrt{-1}\Id) \subset TX \otimes \C$ the holomorphic tangent bundle and by $\dot{R}^L \in \mathrm{End}(T^{(1, 0)}X)$ the endomorphism defined by 
	\begin{equation}
		{R}^L\big(Z, \overline{Z}'\big) = \langle \dot{R}^LZ, \ Z' \rangle, \quad Z, Z' \in T^{(1, 0)}X,
	\end{equation}
we have $b_0 = \det \left(\dot{R}^L/2\pi\right)$. In this paper we will use a refined version of the off-diagonal expansion for the generalized Bergman kernel stated in the work of Kordyukov--Ma--Marinescu~\cite[Theorem 4.3]{KMM19} (see Theorem~\ref{ThBkernelasymptotics} of this paper).

For $f \in L^{\infty}(X)$, the \emph{Toeplitz operator with symbol $f$} is the family of operators acting on $\mH_p([-C_L, C_L])$ defined by 
\begin{equation}
T_{f, p} := P_p M_f P_p,
\end{equation} 
where $M_f$ denotes the multiplication by $f$. Let $W \subset X$ be a compact domain with smooth boundary $V = \partial W$. We consider the Toeplitz operator with symbol the indicator function of $W$
	\begin{equation}
		T_{W, p}:= P_p M_{\bbone_{W}}P_p, \quad \mathrm{with} \quad \bbone_W(x) = \begin{cases}
		1 \quad \mathrm{if} \quad x \in W, \\
		0 \quad  \mathrm{otherwise}.
		\end{cases}
	\end{equation}
Let $\exp_{x_0}^X : T_{x_0}X \rightarrow X$ denote the exponential map for $(X, g^{TX})$ at $x_0 \in X$. Let $\mU_\eta \subset X$ be a tubular neighborhood of $V$ and $V_s \subset \mU_{\eta}$ the hypersurface at distance $s \in ]-\eta/2, \eta/2[$ of $V$, i.e. 

	\begin{equation}\label{hypersurf}
		V_s := \{\exp^X_y(se_{\kn}(y)) \ : \ y \in V \},
	\end{equation}
where $e_{\kn}(y)$ is the \emph{inward pointing unit normal} vector at $y \in V$, as represented in Figure \ref{fig_Vs}.

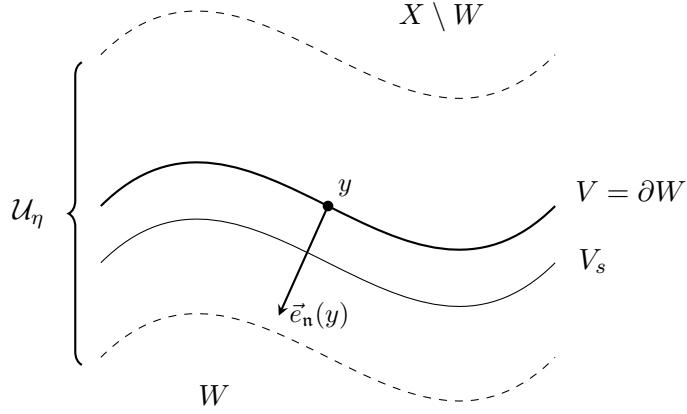
\begin{figure}[H]
\captionsetup{skip = 0pt, aboveskip = 0pt}
\centering
		\begin{tikzpicture}[>=stealth]
			\useasboundingbox (0, -3.2) rectangle (6, 3);

			\draw[dashed] (0,2) .. controls (2,4) and (4,0) .. (6,2);

			\draw[thick] (0,0) .. controls (2,2) and (4,-2) .. (6,0);

			\draw[dashed] (0,-2) .. controls (2,0) and (4,-4) .. (6,-2);

			\draw[] (0,-0.75) .. controls (2,1.25) and (4,-2.75) .. (6,-0.75);

			\node[] at (1.5,-2.5) {$W$};

			\node[] at (4.5,2.5) {$X\setminus W$};

			\node[] at (7, 0.2) {$V = \partial W$};

			\node[] at (6.5, -0.73) {$V_s$};
			
  			\draw [decorate, decoration={brace, amplitude=5pt}, thick] (-0.25,-2.1) -- (-0.25,1.9) node [midway, left=10pt] {$\mU_{\eta}$};

  			\fill (3,0) circle (2pt) node[above right] {\small $y$};

  			\draw[->, thick] (3,0) -- (2.35,-1.45) node[right] {\small $\vec{e}_{\kn}(y)$};
		\end{tikzpicture}
\caption{The case $s \in [0, \eta/2]$.}
\label{fig_Vs}
\end{figure}

As a particular case of Lemma~\ref{Nonlocalizedasymptotics} and Theorem~\ref{ThToeplitzindicasymptoticsFermi}, our first result states as follows.
\begin{thm}\label{ThondiagonalTwexpansion}
Assume $(X, g^{TX})$ is complete, $(X, L, E)$ has bounded geometry, i.e., $\inj^X>0$  and the derivatives of any order of $R^L, R^E, J, g^{TX}$ are uniformly bounded on $X$, and assume \eqref{CourbureUnifPositive}. Let $W$ be a compact domain of $X$ with smooth boundary $V$ and $\mU_{\eta}$ a tubular neighborhood of $V$. Then the following statements hold:
\begin{enumerate}
	\item There exist $p_0 \in \N$, $c > 0$ such that for $l \in \N$, there exist $C_{l, W}, C_{l, W}', C_{l, W}''>0$ such that for any $p \geq p_0$:
	\begin{enumerate}[(i)]
		\item If $x, x' \in X$, 
			\begin{equation}
				|T_{W, p}(x,x')|_{\cC^l(X\times X)} \leq C_{l, W}p^{n+l/2} e^{-c\sqrt{p}d(x,x')}.
			\end{equation}
	\item If $x \in W \setminus V$, 
			\begin{equation}
				|T_{W, p}(x,x)-P_p(x,x)|_{\cC^l(X)} \leq C'_{l, W}p^{n+l/2} e^{-c\sqrt{p}d(x,V)}.
			\end{equation}

	\item If $x \in X \setminus W$,
			\begin{equation}
				|T_{W, p}(x,x)|_{\cC^l(X)} \leq C''_{l, W}p^{n+l/2} e^{-c\sqrt{p}d(x,V)}.
			\end{equation}
	\end{enumerate}

	\item There exist $\mathrm{End}(E)$-valued polynomials $q_{r, x}$ with smooth coefficients in $x$ such that for $x\in \mU_{\eta}$,  if $x \in V_s$ for some $s \in ]-\eta/2, \eta/2[$, 
		\begin{equation}\label{quatre}
			T_{W, p}(x, x) = P_p(x, x)\mathrm{erfc}\left(-s\sqrt{p\pi \alpha_{x}} \right)+ \sum_{r = 1}^{\infty} p^{n-r/2} q_{r, x}(p^{1/2}s) \exp(-s^2p\pi \alpha_{x}) +\mO(p^{-\infty}),
		\end{equation}
with $\alpha_{x} := \omega(e_{\kn}, Je_{\kn})(x)$ and $P_p(x,x)$ has the asymptotics~\eqref{BergmanKondiag}. In the particular case where $J = J_0$, we have $\alpha_{x} = 1$. 

\end{enumerate}
\end{thm}

The \emph{complementary error function} $\mathrm{erfc}$ appearing in~\eqref{quatre} is defined for $x \in \R$ by the formulas

	\begin{equation}\label{erfunction}
		\mathrm{erfc}(x) := \frac{1}{\sqrt{\pi}} \int_{x}^{+\infty} e^{-t^2}dt = \frac{1}{2}-\mathrm{erf}(x),  \quad \mathrm{with} \quad \mathrm{erf}(x) := \frac{1}{\sqrt{\pi}}\int_0^x e^{-t^2}dt,
	\end{equation}
where the function $\mathrm{erf}(x)$ is called the error function. 
The asymptotics we derive provide the following interpretation: in the semiclassical limit, the kernel $T_{W, p}(x, x)$ on the diagonal mimics the behavior of the Bergman kernel for $x$ in the interior of $W$. Conversely, on the complement $X \setminus W$, $T_{W, p}(x, x)$ exhibits exponential decay in $\sqrt{p}$ times the distance to the boundary. Formula~\eqref{quatre} shows that the transition across the interface is universal. After rescaling by \(p^{-1/2}\), the leading profile is always given by the complementary error function, independently of the geometry of \(W\).  The geometry only enters through the local coefficient \(\alpha_x\). Furthermore, when $s \neq 0$, the sum in the right-hand side of~\eqref{quatre} behaves as $\mO(p^{-\infty})$ and at $s = 0$ (where $x \in V$), half-integer powers of $p$ emerge in the asymptotics of $T_{W, p}(x, x)$. 
\begin{rems} We make the following observations: 
\begin{enumerate}[(i)]
\item For convenience, we intentionally left out a factor $2$ in the definition of the functions $\mathrm{erf}$ and $\mathrm{erfc}$ in \eqref{erfunction}.  
\item In our proof we do not really use the compactness assumption on $W$, and the statement of Theorem \ref{ThondiagonalTwexpansion} should hold if we assume that the boundary $V$ admits a tubular neighborhood $\mU_{\eta}$, and has bounded geometry.
\item The kernel asymptotics we derive exhibit behavior similar to that of the partial Bergman kernels as established by Zelditch and Zhou~\cite{ZZ19}, with "$\mathrm{erfc}$" asymptotics along the boundary $V$ (see \cite[Main Theorem]{ZZ19} for comparison).
\end{enumerate}
\end{rems}

We also obtain the following kernel asymptotics for $g(T_{W, p})$ when $g$ is a polynomial vanishing at $0$ and $1$.
\begin{thm}\label{thmNoyauCompo}
	Assume $(X, g^{TX})$ is complete, $(X, L, E)$ has bounded geometry, i.e., $\inj^X>0$  and the derivatives of any order of $R^L, R^E, J, g^{TX}$ are uniformly bounded on $X$, and assume \eqref{CourbureUnifPositive}. Let $W$ be a compact domain of $X$ with smooth boundary $V$ and $\mU_{\eta}$ a tubular neighborhood of $V$. Let $g$ be a polynomial such that $g(0) = g(1) = 0$. 
	\begin{enumerate}
		\item \label{One} There exist $p_0 \in \N$, $c > 0$ such that for $l \in \N$, there exist $C_{l, g, W} >0$ such that for any $p \geq p_0$, if $x \in X \setminus V$, 
		\begin{equation}
			\left|g(T_{W, p})(x, x)\right|_{\cC^l(X)}  \leq C_{l, g, W}p^{n+l/2}e^{-c\sqrt{p}d(x, V)}.
		\end{equation}
		\item \label{Two} There exist smooth $\mathrm{End}(E)$-valued functions  $I_{r, g, x}(s)$, with fast decay of the form $\mO((1+|s|)^{M_{r, g}}e^{-c's^2})$ as $|s| \rightarrow \infty$, which depend smoothly on $x$ and such that for $x \in \mU_{\eta}$, if $x \in V_s$ for some $s \in ]-\eta/2, \eta/2[$,
			\begin{equation}\label{DevOnDiagCompo}
				g(T_{W, p})(x, x) = \sum_{r = 0}^{\infty} p^{n-r/2} I_{r, g, x}(p^{1/2}s) + \mO(p^{-\infty}),
			\end{equation}
		and the above expansion is uniform on $\mU_{\eta}$. Moreover we have 
			\begin{equation}\label{FirstTermDevOnDiagCompo}
				I_{0, g, x}(s) =   \left[\int_{-\infty}^{\infty} g(\mathrm{erfc}(y))\frac{1}{\sqrt{\pi}}e^{-\left(y+s\sqrt{2\pi \alpha_{x}}\right)^2}dy \right] \det\left({\dot{R}^L/2\pi}\right)\mathrm{Id}_{\mathrm{End}(E)},
			\end{equation}
		and $\alpha_x = \omega(e_{\kn}(x), Je_{\kn}(x))$. If $s=0$, i.e. if $x \in V$ in \eqref{DevOnDiagCompo}, we get
		\begin{equation}
			I_{0, g, x}(0) = \left[ \int_0^1 g(y)dy \right]\det\left({\dot{R}^L/2\pi}\right)\mathrm{Id}_{\mathrm{End}(E)}.
		\end{equation}
		In the  particular case where $J = J_0$, we have $\alpha_x = 1$ and $\det({\dot{R}^L/2\pi}) = 1$.
\end{enumerate}
\end{thm}
 Note that the integral term in \eqref{FirstTermDevOnDiagCompo} is the convolution of $g \circ \mathrm{erfc}$ and a Gaussian. Also note that by combining Theorem \ref{ThondiagonalTwexpansion}, Theorem \ref{thmNoyauCompo} and \eqref{BergmanKondiag} we can get the on-diagonal expansion of $g(T_{W, p})$ for $g$ any polynomial.

From the above Theorem, we deduce asymptotics for the trace of $T_{W, p}$. For $g:[0, 1] \rightarrow \R$, let 
	\begin{equation}
		\tilde{g}(x) := g(x)-g(0)(1-x)-g(1)x.
	\end{equation}
We define $I(g)$ by
	\begin{equation}
		I(g) := \int_{-\infty}^{\infty} \tilde{g}\left(\mathrm{erfc}(x)\right)dx.
	\end{equation}
Let $C(V)$ be the quantity defined by 
	\begin{equation}
		C(V) := \int_{V} \bigg[{\det} (J_{0, x_0}\vert_{HV})\alpha_{x_0}\bigg]^{1/2} dv_V(x_0),
	\end{equation}
with $H_{x_0}V = T_{x_0}V \cap J_{x_0}T_{x_0}V$ the maximal complex subspace of $T_{x_0}V$, and $\alpha_{x_0} = \omega(e_{\kn},Je_{\kn})(x_0)$. As a corollary of Theorem \ref{ThondiagonalTwexpansion} and Theorem \ref{thmNoyauCompo}, the third result of this paper states as follows. 

\begin{cor}\label{ThCEGeneralCase}
	Assume $(X, g^{TX})$ is complete, $(X, L, E)$ has bounded geometry, i.e., $\inj^X>0$  and the derivatives of any order of $R^L, R^E, J, g^{TX}$ are uniformly bounded on $X$, and assume \eqref{CourbureUnifPositive}. Let $W$ be a compact domain of $X$ with smooth boundary $V$. Let $g: [0, 1] \rightarrow \R$ be a polynomial vanishing at $0$.
	We have the expansion

	\begin{equation}
	\begin{aligned}
		&\tr(g(T_{W, p})) = \sum_{k \geq 0} p^{n-k/2}a_{k}(g) + \mO(p^{-\infty}), \quad \mathrm{with}\\
		&a_{0}(g) = g(1)\rank(E)\int_W \det \left(\dot{R}^L/2\pi\right) dv_X, \quad a_{1}(g) = (2\pi)^{-1/2}\rank(E)C(V) I(g).
	\end{aligned}
	\end{equation}
	In particular if $J_0 = J$, then $\alpha_{x_0} = 1$, $C(V) = \vol(V)$, $a_0(g) = g(1) \rank(E)\vol(W)$ and $a_1(g) = \rank(E)\vol(V)I(g)$.
\end{cor}

As a consequence of Corollary \ref{ThCEGeneralCase} we obtain the following Weyl law for the Toeplitz operator $T_{W, p}$.

\begin{cor}\label{CorTraceToeplitz}
Assume $(X, g^{TX})$ is complete, $(X, L, E)$ has bounded geometry, i.e., $\inj^X>0$  and the derivatives of any order of $R^L, R^E, J, g^{TX}$ are uniformly bounded on $X$, and assume \eqref{CourbureUnifPositive}. Let $W$ be a compact domain of $X$ with smooth boundary $V$. If $g:[0,1]\rightarrow \R$ is continuous, Hölder continuous at $0$ and $1$ and $g(0) = 0$, we have
	\begin{equation}\label{casegHolder2}
		\tr(g(T_{W, p})) = a_0(g)p^n+ a_1(g)p^{n-1/2} + o(p^{n-1/2}),
	\end{equation}
and for $0<a<b<1$, the concentration of eigenvalues of $T_{W, p}$ in $[a, b]$ verifies
	\begin{equation}\label{repartition_spectre}
		\Big|\spec(T_{W, p})  \cap [a, b] \Big| = p^{n-1/2} \frac{\rank(E)C(V)}{\sqrt{2\pi}} (\mathrm{erfc}^{-1}(a)-\mathrm{erfc}^{-1}(b)) + o(p^{n-1/2}).
	\end{equation}
Moreover, we have, for $0< \vare < 1/2$,
	\begin{equation}\label{SpectrumNear1}
		\Big| \spec(T_{W, p}) \cap [1-\vare, 1]\Big| = p^n \rank(E)\int_W \det \left(\dot{R}^L/2\pi\right) dv_X + o(p^n).
	\end{equation}
In particular if $J_0 = J$, then $C(V) = \vol(V)$ and $\det(\dot{R}^L/2\pi) = 1$.
\end{cor}

The above corollary shows that the operator $T_{W, p}$ asymptotically behaves like a projection: the spectrum of $T_{W, p}$ concentrates near $0$ and $1$ in the semiclassical limit. Note that the kernel of $T_{W, p}$ can be infinite dimensional since $X$ is not necessarily compact.
\begin{rems}
	In the case where $(X, J, \omega, g^{TX})$ is a compact Kähler manifold, $E = \C$, and $(L, h^L)$ is a holomorphic Hermitian prequantum line bundle over $X$, 
	\begin{enumerate}[(i)]
	\item As a particular case of Corollary~\ref{ThCEGeneralCase} and Corollary~\ref{CorTraceToeplitz}, we recover the results established by Charles\textendash Estienne in~\cite[Theorem 1.2, Corollary 1.3]{CE20}. 
	\item The asymptotic behavior of the spectrum of $T_{W, p}$ near $0$ and $1$ has been studied by Berndtsson \cite{Berndtsson03} and Lindholm \cite{Lindholm01}. In particular, we get \eqref{SpectrumNear1} by adapting the proof of \cite[Theorem 3.1]{Berndtsson03}. 
\end{enumerate}
\end{rems}

Our paper is organized as follows. In Section~\ref{SectionBargmann}, we work in the context of the Bargmann space. This setting contains the entire boundary-layer phenomenon that we exhibit. We compute the kernel of the Toeplitz operator whose symbol is the indicator function $\bbone_{\mathbb{H}_s}$ of the half-plane $\mathbb{H}_s := \{z \in \C : \Im(z) > s \}$. We also consider the more general situation where the symbol is the product of $\bbone_{\mathbb{H}_s}$ and a polynomial. The explicit computations performed in Section~\ref{SectionBargmann} provide the local model governing the general kernel asymptotics of Theorem~\ref{ThondiagonalTwexpansion}. In Section~\ref{SectionBergman}, we express in Corollary~\ref{ThBKernelasymptoticsFermi} the  off-diagonal Bergman kernel asymptotics of~\cite[Theorem 4.3]{KMM19} in Fermi coordinates adapted to a tubular neighborhood of the boundary $V$ of $W$. The proofs of our main theorems are contained in Section~\ref{SectionToeplitz}. In \S~\ref{SectionKernToeplitz}, we combine the results of Section~\ref{SectionBargmann} and Section~\ref{SectionBergman} to obtain in Theorem~\ref{ThToeplitzindicasymptoticsFermi} the full off-diagonal kernel asymptotics for $T_{W, p}$. Theorem~\ref{ThondiagonalTwexpansion} is a particular case of this result and of Lemma~\ref{Nonlocalizedasymptotics}. In \S~\ref{PreuveCompoNoyau}, we prove Theorem~\ref{thmNoyauCompo}. Finally, in \S~\ref{tracetoeplitz}, we prove Corollary~\ref{ThCEGeneralCase} and Corollary~\ref{CorTraceToeplitz}. In \S~\ref{Appendix}, we prove a technical result stated in Lemma~\ref{lemmeJq}.
\vspace{0.5cm}
{\paragraph{\textbf{Acknowledgements.}} \mbox{}The author would like to express his gratitude to Professors Xiaonan Ma and George Marinescu, for their guidance, enlightening discussions, and helpful comments. The author is also indebted to Siarhei Finski for the detailed explanations of his work, and to Louis Ioos and the Laboratoire AGM at Cergy Paris Université for their support.}

\section{The case of the Bargmann space \label{SectionBargmann}}

In this section we recall the model case of \cite[\S 1.4]{MM08} for the complex plane and in Sect.~\ref{ToeplitzUpperHalf} we compute the kernel of the Toeplitz operator with symbol $\bbone_{\mathbb{H}_s}$, the indicator function of the \emph{half-plane at distance $s$ of the $x$-axis} $\mathbb{H}_s := \{z \in \C : \Im(z) > s \}$ for $s \in \R$. We then treat the case where the symbol is a product of a polynomial and $\bbone_{\mathbb{H}_s}$. The results are stated in Lemma~\ref{noyauToeplitzInd} and Corollary~\ref{ToeplitzH}. In Sect.~\ref{decayins} we study how the expressions we derive for the kernels behave in terms of the distance to the boundary of $\mathbb{H}_s$. 

This setting is pictured in Figure \ref{fig_demiplan} below and is a good model for what happens at the hypersurface $V_{-s}$ at distance $-s$ of the boundary $V$ of $W$, as in~\eqref{hypersurf}. It will resurface later on in Section \ref{SectionKernToeplitz}. 

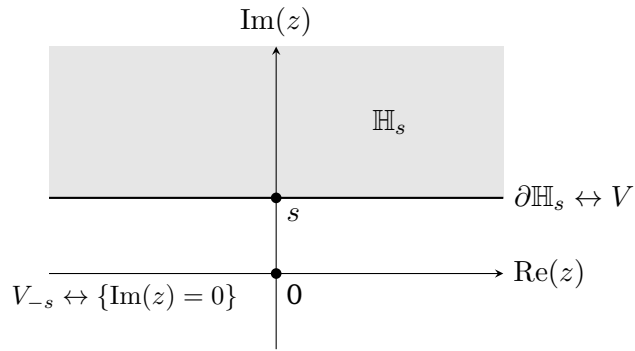
\begin{figure}[H]
	\begin{center}
	\begin{tikzpicture}[>=stealth] 
	\fill[gray!20]
       (-3,1) -- (3,1) -- (3,3) -- (-3,3) -- cycle;
  	\draw[->] (-3,0) -- (3,0) node [right] {$\mathrm{Re}(z)$};

  	\draw[->] (0,-1) -- (0,3) node [above] {$\mathrm{Im}(z)$};

 	\fill (0,0) circle (2pt) node[below right] at (0,0) {0};

	\fill (0, 1) circle (2pt) node [below right] {$s$};
	\draw[thick] (-3, 1)--(3, 1) node [right] {$\partial \mathbb{H}_s \leftrightarrow V$};

	\node[] at (1.5,2) {$\mathbb{H}_s$};

	\node[] at (-2, 0) [below] {\small $V_{-s} \leftrightarrow \{\mathrm{Im}(z) = 0\}$};

	\end{tikzpicture}
	\end{center}
	\caption{Analogy with the general setting, for $s > 0$.}
	\label{fig_demiplan}
\end{figure}

Let us denote by $b_{\alpha}$ and $b_{\alpha}^+$ the complex creation and annihilation operators with domain $L^2(\C)$ and defined by, for $\alpha > 0$,

	\begin{equation}
		b_{\alpha} := -2\frac{\partial}{\partial z} + \pi \alpha \overline{z}, \qquad b_{\alpha}^+ := 2 \frac{\partial}{\partial \overline{z}} + \pi \alpha z.
	\end{equation}
Let $\mL := b_{\alpha}b_{\alpha}^+$ be the harmonic oscillator. An orthonormal basis of $\ker \mL \cap L^2(\C)$ is given by the complex Hermite polynomials 

	\begin{equation}
		h_{k, \alpha}(z) = \alpha^{1/2}\sqrt{\frac{\pi^k \alpha^k}{k!}} z^k \exp\left(-\frac{\pi \alpha}{2}|z|^2\right), \quad k \in \N,
	\end{equation}
and the projector $P_{\alpha} : L^2(\C) \rightarrow \ker \mL$ is a kernel operator with kernel given by
	\begin{equation}
		P_{\alpha}(z, z') = \sum_{k=0}^{+\infty}h_{k, \alpha}(z)\overline{h_{k, \alpha}(z')} = \alpha \exp \left(- \frac{\pi \alpha}{2} (|z|^2 + |z'|^2 - 2z \overline{z}')\right).
	\end{equation}

\subsection{Toeplitz operator with symbol the indicator of the half-plane $\mathbb{H}_s$}\label{ToeplitzUpperHalf}

Let $T_{\alpha}^{s} := P_{\alpha} M_{\bbone_{\mathbb{H}_s}}P_{\alpha}$ be the Toeplitz operator with symbol $\bbone_{\mathbb{H}_s}$ and $T_{\alpha}^{s, c}$ the Toeplitz operator with symbol $\bbone_{\mathbb{H}_s^c}$ where $\mathbb{H}_s^c = \C \setminus \mathbb{H}_s$. The error function $\mathrm{erf}$ introduced in~\eqref{erfunction} extends to $\C$ as the holomorphic function defined by 
				
	\begin{equation}\label{erfholo}
		\mathrm{erf}(z) := \frac{1}{\sqrt{\pi}}\int_{0}^z\exp(-\zeta^2)d\zeta.
	\end{equation}
The above integral does not depend on the smooth path from $0$ to $z \in \C$, and the function $\mathrm{erfc}$ introduced in~\eqref{erfunction}  admits an analytic continuation as well. We have the equalities

	\begin{equation}\label{identiteserf}
		\mathrm{erfc}(z) = \frac{1}{2}-\mathrm{erf}(z),\qquad \mathrm{erf}(-z) = -\mathrm{erf}(z), \qquad \mathrm{erfc}(-z) = 1 - \mathrm{erfc}(z),
	\end{equation}

	\begin{equation}\label{deriver}
		\frac{\partial}{\partial z} \mathrm{erfc}(z) = -\frac{1}{\sqrt{\pi}} e^{-z^2}, \qquad \frac{\partial}{\partial z} \left[z \, \mathrm{erfc}(z)-\frac{1}{2\sqrt{\pi}}e^{-z^2}\right] = \mathrm{erfc}(z).
	\end{equation}

\begin{lem}\label{noyauToeplitzInd}
	For $z, z' \in \C$ the following holds

	\begin{equation}\label{Tnoyaudim21}
		T_{\alpha}^{s}(z, z') = P_{\alpha}(z, z')\mathrm{erfc} \left(\sqrt{\pi \alpha}\left(s - \frac{z-\overline{z}'}{2i}\right) \right),
	\end{equation}

	\begin{equation}\label{Tnoyaudim22}
		T_{\alpha}^{s, c}(z, z')  = P_{\alpha}(z, z')\mathrm{erfc}\left(-\sqrt{\pi \alpha}\left(
			s - \frac{z-\overline{z}'}{2i}\right) \right).
	\end{equation}
\end{lem}

\begin{proof}
		Let $F_a$ and $G$ denote the holomorphic functions defined by, for $a \in \R$,
			
	\begin{equation}
		F_a(z) := \int_{a}^{+\infty} \exp(- \pi \alpha y^2  - \pi \alpha y z)dy, \quad G(z) := \int_{-\infty}^{+\infty} \exp(- \pi  \alpha x^2  - \pi \alpha x z)dx.
	\end{equation}
Notice $F_a$ is the Laplace transform of a Gaussian, and we get for $t \in \R$,
			
	\begin{equation}
		F_a(t) =  \frac{1}{\sqrt{\alpha}}\exp \left(\pi \alpha t^2/4\right) \mathrm{erfc}\left(\sqrt{\pi\alpha}(a+t/2)\right), \quad G(t) = \frac{1}{\sqrt{\alpha}}\exp \left(\pi \alpha t^2/4\right).
	\end{equation}
By the identity theorem, for $z \in \C$ we have 
		
	\begin{equation}
		F_a(z) =  \frac{1}{\sqrt{\alpha}}\exp \left(\pi \alpha z^2/4\right) \mathrm{erfc}\left(\sqrt{\pi \alpha}(a+z/2)\right), \quad G(z) = \frac{1}{\sqrt{\alpha}}\exp \left(\pi \alpha z^2/4\right).
		\label{expressionFaG}
	\end{equation}
By definition we have 

	\begin{equation}
		T_{\alpha}^{s}(z, z') = \int_{\mathbb{H}_s} P_{\alpha}(z, z'')P_{\alpha}(z'', z') d\lambda(z''), 
	\end{equation}
where $\lambda$ denotes the Lebesgue measure. The equality above becomes 
		
	\begin{multline}
		T_{\alpha}^{s}(z, z') = \alpha^2 \exp \left(-\frac{\pi \alpha}{2}(|z|^2 + |z'|^2)\right)\\
		\times \int_{\mathrm{Im}(z'')> s} \exp \left(-\pi \alpha |z''|^2 +  \pi \alpha z\overline{z}'' +  \pi \alpha z'' \overline{z}'\right) d\lambda(z'').
	\end{multline}
By writing $z'' = x'' + iy''$ we get
			
	\begin{equation}
		T_{\alpha}^{s}(z, z') = \alpha^2\exp \left( -\frac{\pi \alpha}{2}(|z|^2 + |z'|^2)\right) G(-(z + \overline{z}'))F_s(i(z-\overline{z}')).
	\end{equation}
By~\eqref{expressionFaG} we get 

	\begin{equation}
		G(-(z + \overline{z}'))F_s(i(z-\overline{z}')) = \alpha^{-1} \exp(\pi \alpha z \overline{z}') \mathrm{erfc} \left(\sqrt{\pi \alpha}\left(s-\frac{z-\overline{z}'}{2i}\right) \right),
	\end{equation}
which proves~\eqref{Tnoyaudim21}. The equality $\mathrm{erfc}(-z) = 1-\mathrm{erfc}(z)$ proves~\eqref{Tnoyaudim22}.
\end{proof}

We now compute the kernel of Toeplitz operators with symbol $\bbone_{\mathbb{H}_s}(z)f(\sqrt{\pi \alpha}z, \sqrt{\pi \alpha}\overline{z})$ where $f$ is a polynomial in $z, \overline{z}$. 
Let us denote by $T^{s}_{\alpha, p, q}$ the Toeplitz operator with symbol \linebreak $z \mapsto {(\sqrt{\pi \alpha}z)}^p {(\sqrt{\pi \alpha}\overline{z})}^q\bbone_{\mathbb{H}_s}$ and $T^{s, c}_{\alpha, p, q}$ defined in the same way for $\bbone_{{\mathbb{H}_s^c}}$. Let $T_{\alpha, p, q}$ be the Toeplitz operator with symbol ${(\sqrt{\pi \alpha}z)}^p {(\sqrt{\pi \alpha}\overline{z})}^q$. Then we have $T^{s}_{\alpha, p, q} + T^{s, c}_{\alpha, p, q} = T_{\alpha, p, q}$. In particular we recover $T^{s}_{\alpha} + T^{s, c}_{\alpha} = P_{\alpha}$.

\begin{lem}
	For $p, q \in \N$ we have 

		\begin{align}
			\begin{split}
				T_{\alpha, p, q}(z, z') &= \alpha \exp\left(- \frac{\pi \alpha}{2}(|z|^2 + |z'|^2)\right)\frac{1}{(\pi \alpha)^{(p+q)/2}}\frac{\partial^p}{\partial \overline{z}'^p}\frac{\partial^q}{\partial z^q}\left[\exp(\pi \alpha z \overline{z}')\right] \\
										&=  (\sqrt{\pi \alpha} z)^p (\sqrt{\pi \alpha} \overline{z}')^qP_{\alpha}(z, z'),
			\end{split}
		\end{align}

		\begin{multline}\label{noyauToeplitzpolynomeH}
			T_{\alpha, p, q}^{s}(z, z') = \alpha\exp \left(-\frac{\pi \alpha}{2}(|z|^2 + |z'|^2)\right) \\
			\times \frac{1}{(\pi \alpha)^{(p+q)/2}}\frac{\partial^p}{\partial \overline{z}'^p}\frac{\partial^q}{\partial z^q}\left[\exp(\pi \alpha z \overline{z}') \mathrm{erfc} \left(\sqrt{\pi \alpha}\left(s-\frac{z-\overline{z}'}{2i}\right)\right)\right].
		\end{multline}
	We also get 
	
		\begin{multline}\label{noyauToeplitzpolynomeHc}
			T^{s, c}_{\alpha, p, q}(z,z') = \alpha\exp \left(-\frac{\pi \alpha}{2}(|z|^2 + |z'|^2)\right) \\
			\times \frac{1}{(\pi \alpha)^{(p+q)/2}}\frac{\partial^p}{\partial \overline{z}'^p}\frac{\partial^q}{\partial z^q}\left[\exp(\pi \alpha z \overline{z}') \mathrm{erfc} \left(-\sqrt{\pi \alpha}\left(s-\frac{z-\overline{z}'}{2i}\right)\right)\right].
		\end{multline}
\end{lem}

\begin{proof}
As in the proof of Lemma~\ref{noyauToeplitzInd}, we have 

		\begin{multline}
			T_{\alpha, p, q}^{s}(z, z') = \exp \left(-\frac{\pi \alpha}{2}(|z|^2 + |z'|^2)\right)\\ \times \int_{\mathrm{Im}(z'')> s}\alpha^2 (\pi \alpha)^{(p+q)/2} z''^p \overline{z}''^q\exp \left(-\pi \alpha |z''|^2 +  \pi \alpha z\overline{z}'' +  \pi \alpha z'' \overline{z}'\right) d\lambda(z'').
		\end{multline}
The above equation writes as 

		\begin{multline}
			T_{\alpha, p, q}^{s}(z, z') = \alpha^2\exp \left(-\frac{\pi \alpha}{2}(|z|^2 + |z'|^2)\right) \\ 
			\times \frac{1}{(\pi \alpha)^{(p+q)/2}}\frac{\partial^p}{\partial \overline{z}'^p}\frac{\partial^q}{\partial z^q}\int_{\mathrm{Im}(z'')> s}\exp \left(-\pi \alpha |z''|^2 +  \pi \alpha z\overline{z}'' +  \pi \alpha z'' \overline{z}'\right) d\lambda(z'').
		\end{multline}
We did compute the integral above in the proof of Lemma~\ref{noyauToeplitzInd}, which proves~\eqref{noyauToeplitzpolynomeH}.  The same method proves~\eqref{noyauToeplitzpolynomeHc}.
\end{proof}

The following corollary is a consequence of~\eqref{deriver},~\eqref{noyauToeplitzpolynomeH} and~\eqref{noyauToeplitzpolynomeHc}. 

\begin{cor}\label{ToeplitzH}
	Let $f$ be a polynomial in $\sqrt{\pi \alpha}z, \sqrt{\pi \alpha}\overline{z}$ and $T_{\alpha, f}, T^{s}_{\alpha, f}$ and $T^{s, c}_{\alpha, f}$ the Toeplitz operators with symbols $f, \bbone_{\mathbb{H}_s}f$ and $\bbone_{\mathbb{H}_s^c}f$. Then there exists a polynomial $Q(f)$ of degree less than $\deg f$ such that 

	\begin{equation}\label{Toeplitzpolynome}
		T_{\alpha, f}(z, z') = f(\sqrt{\pi\alpha}z, \sqrt{\pi\alpha} \overline{z}')P_{\alpha}(z, z'),
	\end{equation}

	\begin{multline}\label{ToeplitzpolynomeH}
		T^{s}_{\alpha, f}(z, z') = T_{\alpha, f}(z, z')\mathrm{erfc} \left(\sqrt{\pi \alpha}\left(s-\frac{z-\overline{z}'}{2i}\right)\right) \\
		+ Q(f)(\sqrt{\pi\alpha}z, \sqrt{ \pi\alpha} \overline{z}', \sqrt{\pi \alpha}s)P_{\alpha}(z, z') \exp \left(-\pi \alpha\left(s-\frac{z-\overline{z}'}{2i}\right)^2\right),
	\end{multline}

	\begin{multline}\label{ToeplitzpolynomeHc}
		T^{s, c}_{\alpha, f}(z, z') = T_{\alpha, f}(z, z')\mathrm{erfc} \left( -\sqrt{\pi \alpha}\left(s-\frac{z-\overline{z}'}{2i}\right)\right) \\
		- Q(f)(\sqrt{\pi\alpha}z, \sqrt{ \pi\alpha} \overline{z}', \sqrt{\pi \alpha}s)P_{\alpha}(z, z') \exp \left(-\pi \alpha\left(s-\frac{z-\overline{z}'}{2i}\right)^2\right).
	\end{multline}

\end{cor}

\begin{rem}
	The formula \eqref{Toeplitzpolynome} is a particular case of the functional calculus developed by Ma--Marinescu as in \cite[\S 7.1]{MM07}. In our case we will use the formulas \eqref{ToeplitzpolynomeH} and \eqref{ToeplitzpolynomeHc}. The error function term $\mathrm{erfc}$ appearing after integration over the half plane $\mathbb{H}_s$ poses a challenge when trying to compute the kernel of the composition $T_{\alpha, f}^s \circ T_{\alpha, g}^s$, with $f, g$ polynomials in $\sqrt{\pi \alpha}z$, $\sqrt{\pi \alpha}\overline{z}$. Therefore, the presence of this error-function profile prevents a direct application of the usual Toeplitz symbolic calculus. Instead, we develop an approach based on explicit kernel representations.
\end{rem}

\subsection{Exponential decay in $s$}\label{decayins}

From the previous computations, we notice that the kernels of the Toeplitz operators $T_{\alpha, f}^s, T_{\alpha, f}^{s, c}$ have a form of exponential decay in $s$ at $(0,0)$, which we state more precisely as an inequality in Corollary~\ref{ineqToeplitzBargmann}. 

As in~\cite{S51}, for $z = x + iy$, by computing the integral in~\eqref{erfholo} along the paths $\Gamma_1(u) = u, \quad u \in [0, x]$ and $\Gamma_2(u) = x + iu, \quad u \in [0, y]$ we get
		
	\begin{equation}
		\mathrm{erf}(z) = \mathrm{erf}(x) + e^{-x^2} \int_0^y e^{u^2} \sin(2xu)du + i e^{-x^2} \int_0^y e^{u^2} \cos(2xu)du.
	\end{equation}
By~\eqref{identiteserf}, we obtain 

	\begin{equation}
		\mathrm{erfc}(z) = \mathrm{erfc}(x) - e^{-x^2} \int_0^y e^{u^2} \sin(2xu)du - i e^{-x^2} \int_0^y e^{u^2} \cos(2xu)du.
	\end{equation}
For $x > 0$, since $0 \leq \mathrm{erfc}(x) \leq e^{-x^2}$ we have the inequality

	\begin{equation}
		|\mathrm{erfc(z)}| \leq e^{-x^2}\left(1 + 2|y|e^{y^2}\right), 
	\end{equation}
which gives if $z = x + iy$, $z' = x' + iy'$ and $y + y' \leq 2s$

	\begin{multline}
		\left|\mathrm{erfc}\left( \sqrt{\pi \alpha}\left(s-\frac{z-\overline{z}'}{2i}\right)\right)\right|  \\
		\leq\left[1 + \sqrt{\pi \alpha}|x-x'|\exp\left(\frac{\pi \alpha}{4}(x-x')^2 \right)\right]\exp \left(-\pi \alpha \left(s- \frac{y + y'}{2}\right)^2\right),
	\end{multline}
on the other hand if $x \leq 0$, $\mathrm{erfc}(-x) \leq e^{-x^2}$, therefore if $y + y' \geq 2s$,

	\begin{multline}
		\left|\mathrm{erfc}\left(- \sqrt{\pi \alpha}\left(s-\frac{z-\overline{z}'}{2i}\right)\right)\right|  \\
		\leq\left[1 + \sqrt{\pi \alpha}|x-x'|\exp\left(\frac{\pi \alpha}{4}(x-x')^2 \right)\right]\exp \left(-\pi \alpha \left(s- \frac{y + y'}{2}\right)^2\right).
	\end{multline}
We also have the following inequality on the Bergman kernel 

	\begin{equation}
		|P_{\alpha}(z, z')| \leq \alpha \exp \left(-\frac{\pi \alpha}{2}(x-x')^2 -\frac{\pi \alpha}{2} (y - y')^2 \right),
	\end{equation}
which gets us if $y+y' \leq 2s$ 

	\begin{equation}
		|T_{\alpha}^s(z,z')| \leq \alpha \left(1 + \sqrt{\pi \alpha}|z-z'|\right)\exp \left(- \frac{\pi \alpha}{4}|z-z'|^2  \right)\exp \left(\frac{\pi \alpha}{4} \left(s- y + s- y'\right)^2\right),
	\end{equation}
and if $y + y' \geq 2s$, 
	\begin{equation}
		|T_{\alpha}^{s, c}(z,z')| \leq \alpha \left(1 + \sqrt{\pi \alpha}|z-z'|\right)\exp \left(- \frac{\pi \alpha}{4}|z-z'|^2  \right)\exp \left(-\frac{\pi \alpha}{4} \left(s- y + s- y'\right)^2\right).
	\end{equation}
In the same way by Corollary~\ref{ToeplitzH}  we get the following result.

\begin{prop}
	For $f$ polynomial there exist $C > 0$, $M_{f} \geq 0$ such that for $\alpha, \alpha' \in \N^2$, if $y+ y' \leq 2s$, 
		\begin{multline}
			\left|T_{\alpha, f}^s(z,z')\right| \leq \alpha C\left(1 + \sqrt{\pi \alpha}|z|+ \sqrt{\pi \alpha}|z'|+\sqrt{\pi \alpha}|s|\right)^{M_{f}} \\
			\times \exp \left(- \frac{\pi \alpha}{4}|z-z'|^2  \right)\exp \left(-\frac{\pi \alpha}{4} (s-y + s-y')^2\right),
		\end{multline}
	If $y + y' \geq 2s$, 
		\begin{multline}
			\left|T_{\alpha, f}^{s, c}(z,z')\right| \leq \alpha C\left(1 + \sqrt{\pi \alpha}|z|+ \sqrt{\pi \alpha}|z'|+\sqrt{\pi\alpha} |s|\right)^{M_{f}} \\
			\times \exp \left(- \frac{\pi \alpha}{4}|z-z'|^2  \right)\exp \left(-\frac{\pi \alpha}{4} (s-y + s-y')^2\right).
		\end{multline}
\end{prop}

We can simplify the above proposition to state the following corollary.

\begin{cor}\label{ineqToeplitzBargmann}
	If $s \geq 0$, $\Im(z) \leq 0, \Im(z') \leq 0$,

	\begin{equation}
		\left|T_{\alpha, f}^s(z,z')\right| \leq \alpha Ce^{-\pi \alpha s^2}\left(1 + \sqrt{\pi \alpha}|z|+ \sqrt{\pi \alpha}|z'|+\sqrt{\pi\alpha} |s|\right)^{M_{f}} \exp \left(- \frac{\pi \alpha}{4}|z-z'|^2  \right).
	\end{equation}
If $s \leq 0$, $\Im(z) \geq 0, \Im(z') \geq 0$,

	\begin{equation}
		\left|T_{\alpha, f}^{s,c}(z,z')\right| \leq \alpha Ce^{-\pi \alpha s^2}\left(1 + \sqrt{\pi \alpha}|z|+ \sqrt{\pi \alpha}|z'|+\sqrt{\pi \alpha} |s|\right)^{M_{f}} \exp \left(- \frac{\pi \alpha}{4}|z-z'|^2  \right).
	\end{equation}
\end{cor}

\begin{rem}
	We could still get a version of Corollary~\ref{ineqToeplitzBargmann} that holds for all values of $z, z'$ by replacing the inequality $0 \leq \mathrm{erfc}(x) \leq e^{-x^2}$ by $0 \leq \mathrm{erfc}(x) \leq 1 + e^{-x^2}$ in the proof, however we would lose the decay in $s$ by doing so. Note that the inequality we obtain may not be preserved under the composition of Toeplitz operators.
\end{rem}

\section{Bergman kernel asymptotics}\label{SectionBergman}

This section is organized as follows. In Sect.~\ref{AsymptoticsKMM} we enunciate the off-diagonal expansion of the generalized Bergman kernel obtained by Kordyukov\textendash Ma\textendash Marinescu in \cite{KMM19}, which is a more detailed statement than the expansion~\eqref{BergmanKondiag} given in the introduction. In Sect.~\ref{GenBergmanFermi} we give the expression of the Bergman kernel in so-called Fermi coordinates adapted to a tubular neighborhood of the boundary $V$ of $W$. In Sect.~\ref{AsymptoticsDiffeoFermi} we compute the asymptotic expansion of the correspondence between Fermi and normal coordinates. In Sect.~\ref{AsymptoticsGenBergmanFermi} we derive asymptotics for the generalized Bergman kernel in Fermi coordinates from the off-diagonal expansion described Sect.~\ref{AsymptoticsKMM}. The technique is adapted from Finski~\cite[\S 5.2]{Finski24}. In the rest of the paper we will also use the term ``Bergman kernel'' to refer to the generalized Bergman kernel. 

We quickly recall the context of \cite{KMM19} as in the introduction. Let $(X, \omega)$ be a symplectic manifold of dimension $m=2n$ endowed with $(L, h^L)$ a Hermitian line bundle and $(E, h^E)$ a Hermitian vector bundle, with Hermitian connections $\nabla^L$ and $\nabla^E$ whose curvatures are denoted $R^L$ and $R^E$. We assume $L$ satisfies the prequantization condition \eqref{prequantum}. 

We endow $(X, \omega)$ with a Riemannian metric $g^{TX}$ with Levi-Civita connection $\nabla^{TX}$, and curvature $R^{TX}:= (\nabla^{TX})^2$. Let $J:TX \rightarrow TX$ be an almost-complex structure on $X$ compatible with $\omega$ and $g^{TX}$, that is $\omega(J \cdot, J \cdot) = \omega(\cdot, \cdot)$, $g^{TX}(J \cdot, J \cdot) = g^{TX}(\cdot, \cdot)$ and $\omega(u, Ju) > 0$ for any $u \in TX, u \neq 0$. Let $J_0: TX \rightarrow TX$ be the skew-adjoint linear map satisfying $\omega(u, v) =g^{TX}(J_0u, v)$ for all $u, v \in TX$. Then we have $J = J_0(-J_0^2)^{-1/2}$, and $J_0$ commutes with $J$  (see \cite[Introduction]{MM08}). 

We assume $X$ has \emph{bounded geometry}, i.e., $\inj^X > 0$ and the derivatives of any order of $R^L, R^E, J, g^{TX}$ are uniformly bounded on $X$. We also assume 
\begin{equation}
		\mu_0 := \underset{x \in X, u \in T_xX \setminus \{0\}}{\inf} {\frac{\sqrt{-1} R_x^L(u, Ju)}{|u|_{g_x^{TX}}^2}} > 0.
\end{equation}

Let $\Delta^{L^p \otimes E}$ be the \emph{Bochner Laplacian} acting on sections of $L^p \otimes E$ associated with the connections $\nabla^{TX}, \nabla^L, \nabla^E$ and defined by the formula 

\begin{equation}\label{BochnerLapl}
	\Delta^{L^p \otimes E} := - \sum_{i=1}^m \left[ (\nabla^{L^p \otimes E}_{e_i})^2 - \nabla^{L^p \otimes E}_{\nabla^{TX}_{e_i}e_i} \right],
\end{equation}
where $(e_i)_{i \in \{1, \ldots, m\}}$ is a local orthonormal frame of $TX$ and $\nabla^{L^p \otimes E}$ denotes the connection on $L^p \otimes E$ associated with $\nabla^L, \nabla^E$. Then the \emph{renormalized Bochner Laplacian} defined as in \eqref{RenormBochnerL} has spectral gap \eqref{rBochnerspectralgap}. For $p$ large enough, we define the \emph{generalized Bergman projector} $P_p$ as in \eqref{Bergmanprojector}. By the Schwartz kernel theorem, it is a kernel operator with respect to the Riemannian volume form $dv_X$ of $(X, g^{TX})$ with kernel the \emph{generalized Bergman kernel} $P_p( \cdot, \cdot) \in \cC^{\infty}(X \times X, (L^p \otimes E) \boxtimes {(L^p \otimes E)}^*)$. As in \cite[Remark 1.4.3, Problem 1.5]{MM07}, the smoothness of $P_p$ comes from the ellipticity of the operator $\Delta_p$. 

\begin{rems}The asymptotic expansion of the Bergman kernel has a long history:
\begin{enumerate}
	\item[(i)] In the case where $(X, J, \omega, g^{TX})$ is a compact Kähler manifold, $E = \C$ and $(L, h^L)$ is a holomorphic Hermitian prequantum line bundle over $X$,  the renormalised Bochner Laplacian $\Delta_p$ defined in \eqref{RenormBochnerL} coincides with twice the Kodaira Laplacian $\square_p = \overline{\partial}^{L^p *} \overline{\partial}^{L^p}$, and $\mH_p([-C_L, C_L])$ coincides with $\Ker \square_p = H^0(X, L^p)$ the space of holomorphic sections of $L^p$. In that case, the kernel of the projector $P_p$ is called the Bergman kernel, $b_0 = \Id$, and the diagonal expansion~\eqref{BergmanKondiag} is a well known result from the works of Tian~\cite{Tian}, Catlin~\cite{Catlin} and Zelditch~\cite{Zeld}, Dai--Liu--Ma~\cite{DLM06}, Ma--Marinescu~\cite{MM08}, with many other contributions.
	\item[(ii)] When $(X,\omega)$ is compact symplectic and $P_p$ is the projection on the kernel of the $\mathrm{spin}^c$-Dirac operator, the full off-diagonal expansion of the Bergman kernel $P_p$ was established by Dai\textendash Liu\textendash Ma \cite[Theorem 4.18]{DLM06}. The expansions of Ma\textendash Marinescu and Kordyukov--Ma--Marinescu for the generalized Bergman kernel in the noncompact case are of the same nature.
\end{enumerate}
\end{rems}
\subsection{Off-diagonal asymptotics of the generalized Bergman kernel}\label{AsymptoticsKMM}

In~\cite[Theorem 5]{DLM06}, Dai--Liu--Ma established a full off-diagonal expansion for the Bergman kernel of the $\mathrm{spin}^c$-Dirac operator. In~\cite[Theorem 1.19]{MM08}, Ma--Marinescu obtained an expansion for the generalized Bergman kernel in a neighborhood of size $\vare/\sqrt{p}$ of the diagonal, called near off-diagonal expansion. This result has been improved by Kordyukov--Ma--Marinescu in~\cite[Theorem 4.3]{KMM19} to an off-diagonal expansion similar to the one of Dai--Liu--Ma. We recall the result of Kordyukov--Ma--Marinescu in this paragraph.

Let $x_0 \in X$, $\vare < \inj^X$ and $B^X(x_0, \vare)$ be a normal coordinate ball centered at $x_0$ of radius $\vare$. We identify elements of $B^X(x_0, \vare)$ and $B^{T_{x_0}X}(0, \vare)$ by the diffeomorphism given by the exponential map $\exp^X_{x_0} : B^{T_{x_0}X}(0, \vare) \rightarrow B^X(x_0, \vare)$. In the rest of the paper we denote $\varp_{x_0}(Z) :=  \exp^X_{x_0}(Z)$. For $Z \in B^{T_{x_0}X}(0, \vare)$, the fibers $E_Z$, $L_Z$, $T_ZX$ and their duals, tensor products and exterior products are identified to $E_{x_0}, L_{x_0}$ and $T_{x_0}X$ by parallel transport along the curve $[0, 1] \ni t \mapsto tZ$ for the connections $\nabla^E$, $\nabla^L$ and $\nabla^{TX}$ and the connections they induce. Under this identification there exists a function $\kappa_{x_0}$ such that the localized volume form $dv_{X_0}$ verifies

	\begin{equation}
		dv_{X_0}(Z) = \kappa_{x_0}(Z)dv_{T_{x_0}X}(Z), \quad \mathrm{with} \quad \kappa_{x_0}(0) = 1,
	\end{equation} 
and where $dv_{T_{x_0}X}(Z)$ denotes the Riemannian volume form of $(T_{x_0}X, g^{T_{x_0}X})$. Still under this identification, the localized Bergman kernel $P_{p, x_0}(Z, Z')$ becomes a smooth function mapping $Z, Z'\in B^{T_{x_0}X}(0, \vare)$ to an element of $(L^p \otimes E)_{x_0} \otimes (L^p \otimes E)^*_{x_0} \simeq \C \otimes \End(E)_{x_0}$.
Let $\mJ_{x_0}$ be the endomorphism
	
	\begin{equation}\label{calJ0}
		\mJ_{x_0} = -2\pi \sqrt{-1} J_0(x_0),
	\end{equation}
then $\mJ_{x_0}: T^{(1, 0)}_{x_0} X \rightarrow T^{(1, 0)}_{x_0} X$ is positive and $\mJ_{x_0} : T_{x_0}X \rightarrow T_{x_0} X$ is skew-adjoint. Let $\mP = \mP_{x_0} \in \cC^{\infty}(T_{x_0} X \times T_{x_0} X)$ be the kernel defined by
	
	\begin{equation}\label{calP}
		\mP_{x_0}(Z, Z') := \frac{\det_{\C} \mJ_{x_0}}{(2\pi)^n} \exp \left( -\frac{1}{4} \scal{(\mJ_{x_0}^2)^{1/2}(Z-Z')}{(Z-Z')} + \frac{1}{2}\scal{\mJ_{x_0} Z}{Z'}\right).
	\end{equation}
It corresponds to the Bergman kernel of the operator $\mL$ on $\C^n$. By~\cite[Theorem 4.3]{KMM19}, we have the following off-diagonal expansion for the generalized Bergman kernel, as an analogue of \cite[Theorem 4.18]{DLM06}.

\begin{thm}[{\cite[Theorem 4.3]{KMM19}}]\label{ThBkernelasymptotics}
	There exists $\vare \in ]0, \inj^X[$ such that for any $j, l, l' \in \N$, there exist $C, c_0, M > 0$ such that for any $p \geq 1, x_0 \in X$ and $Z, Z' \in T_{x_0} X, |Z|, |Z'| < \vare$ we have
			 
		\begin{multline}
		        \sup_{|\alpha| + |\alpha'| \leq l}\left| \frac{\partial^{|\alpha| + |\alpha'|}}{\partial Z^{\alpha} \partial Z'^{\alpha'}} \left( p^{-n} P_{p, x_0}(Z, Z') -  \sum_{r=0}^j \mF_{r, x_0}( \sqrt{p}Z, \sqrt{p}Z') \kappa_{x_0}^{-1/2}(Z) \kappa_{x_0}^{-1/2}(Z') p^{-r/2} \right) \right|_{\cC^{l'}(X)} \\
		       \leq C p^{-\frac{j-l+1}{2}} (1 + \sqrt{p}|Z| + \sqrt{p}|Z'|)^{M} \exp(-\sqrt{\mu_0 p}|Z-Z'|) + \mO(e^{-c_0 \sqrt{p}}), \\
		\end{multline}
with $| \ . \ |_{\cC^{l'}(X)}$ the $\cC^{l'}$ norm for the parameter $x_0 \in X$,  

		\begin{equation}
			\mF_{r, x_0}(Z, Z') := J_{r, x_0}(Z, Z')\mP_{x_0}(Z, Z'), \quad J_{0, x_0}(Z, Z') = \Id_{E_{x_0}},
		\end{equation}	
the terms $J_{r, x_0}(Z, Z')$ are $\mathrm{End}(E)_{x_0}$-valued polynomials in $Z, Z'$ with the same parity as $r$, and verify $\deg J_{r, x_0}(Z, Z') \leq 3r$, their coefficients are polynomials in $R^{TX}, R^E, R^L$ and their derivatives of order $\leq r-2$ at $x_0$ and reciprocals of linear combinations of eigenvalues of $\mJ_{x_0}$.
\end{thm}

\subsection{Generalized Bergman kernel in Fermi coordinates}\label{GenBergmanFermi}

In this paragraph we adapt the proofs of Finski in~\cite[\S 5.2]{Finski24} and give in \eqref{BergmanFermi} the expression of the generalized Bergman kernel in Fermi coordinates in terms of its expression in normal coordinates and of the diffeomorphism between normal and Fermi coordinates. The difference in our case stems from the fact that the horizontal bundle and the normal bundle are odd-dimensional. We start by introducing some notations.

Let $W \subseteq X$ be a compact domain of $X$ with boundary $V = \partial W$. Since $W$ is compact, $V$ is also compact and the injectivity radius $\mathrm{inj}^V$ of $V$ is positive. We have a splitting of the tangent bundle $TX \vert_V = T V \oplus N$ i.e.,
			
	\begin{equation}\label{split}
		T_yX = T_y V \oplus N_y, \quad y \in V,
	\end{equation}
where $N_y$ is the orthogonal complement to $T_y V$ with respect to $g_y^{TX}$, and $N \rightarrow V$ identifies with the normal bundle to $V$.  Let $e_{\kn}(y) \in N_y$ be \emph{the inward pointing unit normal} at $y \in V$. By compactness of $V$ there exists $\eta > 0$ such that for all $y \in V$, the geodesic $ ]-\eta, \eta[ \ni u \mapsto  \exp^X_y(ue_{\kn}(y))$ is well defined. This identifies $V \times ]-\eta, \eta[$ to $\mU_{\eta}$ a tubular neighborhood of $V$. For $s \in ]-\eta/2, \eta/2[$, denote by $V_s$ the hypersurface defined by 

	\begin{equation}
		V_s := \{\exp^X_y(se_{\kn}(y)) \ : \  y \in V \},
	\end{equation}			
in particular for $s > 0, V_s = \{y \in W : d(y, V) = s\}$. As in \cite{BM06}[\S 1.2], we extend the splitting \eqref{split} to $\mU_{\eta}$ by
	\begin{equation}
		T_y X = T_y V_s \oplus N_{s, y}, \quad y \in V_s,
	\end{equation} 
where $N_{s, y}$ denotes the normal bundle to $V_s$, such that we now have

	\begin{equation}
		\begin{gathered}
			TX \vert_{\mU_{\eta}} = TV \oplus N \quad \mathrm{with} \quad TV \rightarrow \mU_{\eta}, \quad N \rightarrow \mU_{\eta},
		\end{gathered}
	\end{equation}
and there exists a family of metrics $g^{TV}$ such that
	\begin{equation}
		g^{TX}(y, x_m) = g^{TV}(y) + dx_m^2, \quad (y, x_m) \in V \times (-\eta, \eta).  
	\end{equation}
Let $P^N : TX \vert_{\mU_{\eta}} \rightarrow N, \ P^{V} : TX \vert_{\mU_{\eta}} \rightarrow TV$ be the orthogonal projections induced by the above splitting and introduce the connections $\nabla^N := P^N \nabla^{TX}$, $\nabla^{TV}$ the Levi-Civita connection for $g^{TV}$ and $\nabla^{sp} := \nabla^{TV} \oplus \nabla^N$ the split connection on $TX \vert_{\mU_{\eta}}$ induced by $\nabla^{TV}$ and $\nabla^N$ and $R^{sp}$ its curvature. The \emph{second fundamental form} of $V$ is defined by 
			
	\begin{equation}
		A := \nabla^{TX} -  \nabla^{sp} \in \Omega^1(\mU_{\eta}, \End(TX \vert_{\mU_{\eta}})).
	\end{equation}
Denote by $B^X(x_0, \vare) \subseteq X$ the normal coordinate ball in $X$ centered at $x_0$ of radius $\vare < \inj^X$. If $x_0 \in V_s$ and $\vare' < \inj^V/2$, let $B^{V_s}(x_0, \vare')$ be the normal coordinate ball in $V_s$ centered at $x_0$ of radius $\vare$ and $\exp^{V_s}_{x_0}$ the corresponding exponential map. Let  $U_{\eta, \vare'}(x_0) := B^{T_{x_0}V_s}(0, \vare') \times ]-\eta-s, \eta-s[$ and define $\psi_{x_0}$ by
			
	\begin{equation}
		\begin{gathered}
			\psi_{x_0}(Y, x_m) :=  \Psi_s(\exp_{x_0}^{V_s} Y, x_m), \quad (Y, x_m) \in U_{\eta, \vare'}(x_0), \quad \mathrm{where}\\
			\Psi_s : (y, x_m) \in V_s \times ]-\eta-s, \eta-s[ \mapsto \exp^X_{y} (x_m e_{\kn}(y)) \in {\mU}_{\eta}.
		\end{gathered}
	\end{equation}
The map $\psi_{x_0}$ defines a diffeomorphism $\psi_{x_0} : U_{\eta, \vare'}(x_0) \rightarrow {\mU}_{\eta, \vare'}(x_0)$, where \linebreak ${\mU}_{\eta, \vare'}(x_0) := \Psi_s(B^{V_s}(x_0, \vare') \times ]-\eta-s, \eta-s[)$ is a portion of the tubular neighborhood ${\mU}_{\eta}$ containing $x_0$. 
Let $h_s :   U_{\eta, \vare'}(x_0) \rightarrow B^{T_{x_0}X}(0, \vare)$ be the map such that 
			
	\begin{equation}\label{psiexph}
		\psi_{x_0}(Y, x_m) = \exp^X_{x_0}(h_s(Y,x_m)) = \varp_{x_0}(h_s(Y, x_m)).
	\end{equation}
Then $h_s$ is a diffeomorphism onto its image. We have the following diagram, where $\iota$ is the inclusion map :
			
	\begin{center}
		\begin{tikzpicture}
  		\matrix (m) [matrix of math nodes,row sep=3em,column sep=4em,minimum width=2em]
  			{
     		B^X(x_0, \vare) & B^{T_{x_0}X}(0, \vare) \\
     		\mU_{\eta, \vare'}(x_0) &  U_{\eta, \vare'}(x_0) \\};
  			\path[-stealth]
    		(m-2-1) edge node [left] {$\iota$} (m-1-1)
            (m-1-2) edge node [below] {$\varp_{x_0}(Z)$} (m-1-1)
    		(m-2-2) edge node [right] {$h_s(Y, x_m)$} (m-1-2)
    		(m-2-2) edge node [below] {$\psi_{x_0}(Y, x_m)$} (m-2-1);
		\end{tikzpicture}
	\end{center}
We picture the situation as follows.
\begin{figure}[H]
	\begin{center}
		\begin{tikzpicture}[>=stealth]

	\fill[gray!20] (-1,2) rectangle (1,-2);
				
	\draw[-] (-3,0) -- (3,0) node [right] {$V= V_0$};
		
	\fill (0, 1) circle (2pt) node [below] {$x_0$};

	\draw[dashed] (-3, 2) -- (3, 2) node [right] {$\eta$};

	\draw[dashed] (-3, -2) -- (3, -2) node [right] {$-\eta$};

	\draw (-3, 1)--(3, 1) node [right] {$V_s$};

	\draw[->] (-3.5, -2.5)--(-3.5, 2.5);

	\draw[thick] (-1, 1)--(1, 1) node [midway, above] {\scriptsize $B^{V_s}(x_0, \vare')$};

	\fill (-3.5, 0) circle (2pt) node [left] {$-s$};

	\fill (-3.5, 2) circle (2pt) node [left] {$\eta -s$};

	\fill (-3.5, -2) circle (2pt) node [left] {$-\eta -s$};

	\fill (-3.5, 1) circle (2pt) node [left] {$0$};

	\fill[pattern=north east lines]
       (-1,0.95) rectangle (1,1.05);

	\node[below] at (0,-1) {$\mU_{\eta, \vare'}(x_0)$};

	\end{tikzpicture}
	\end{center}
	\caption{In grey, the neighborhood $\mU_{\eta, \vare'}(x_0)$.}
\end{figure}

We identify elements of $U_{\eta, \vare'}(x_0)$ and $\mU_{\eta, \vare'}(x_0)$ by the diffeomorphism $\psi_{x_0}$ and for $(Y, x_m) \in U_{\eta, \vare'}(x_0)$ the fibers $E_{h_s(Y, x_m)}, L_{h_s(Y, x_m)}$ and $T_{h_s(Y, x_m)}X$ to $E_{x_0}, L_{x_0}$ and $T_{x_0}X$ by parallel transport along the curve 
	\begin{equation}\label{defdelta}
		\delta(t) = \psi_{x_0}(tY, 0), \quad t \in [0,1], \qquad \delta(t) = \psi_{x_0}(Y, (t-1)x_m), \quad t \in [1,2],
	\end{equation}
for the connections $\nabla^E, \nabla^L, \nabla^{TX}$. The same goes for tensor products, duals and exterior products of these vector bundles with respect to the naturally induced connections. Under this identification the Bergman kernel at $\psi_{x_0}(Y, x_m)$ and  $\psi_{x_0}(Y', x_m')$ becomes $P_{p, x_0}((Y, x_m), (Y', x_m'))$ an element of $\End(L^p \otimes E)_{x_0} \simeq \C \otimes \End(E)_{x_0}$.  This identification can be applied to any function in $\cC^{\infty}(X \times X, F \otimes F^*)$ where $F$ is a Hermitian vector bundle over $X$. 

Let $(F, h^F)$ be a Hermitian vector bundle over $X$ with Hermitian connection $\nabla^F$. Let $x_0 \in V_s$ and $f_i$ a Hermitian basis of $F_{x_0}$. Let $\tilde{f}_i$ be the basis of $F_{Z}$ obtained by parallel transport along the curve $\gamma : t \in [0, 1] \mapsto tZ$ for $\nabla^F$ and $\tilde{f}'_i$ the basis of $F_{h_s(Y, x_m)}$ obtained by parallel transport along $\delta$ for $\nabla^F$.  Let $\theta_s^F:B^{T_{x_0}X}(0, \vare) \rightarrow M_{\rank{F}}(\C)$ be the matrix-valued function defined by
		
	\begin{equation}\label{matricepassage}
		\tilde{f_i}(h_s(Y, x_m)) = \sum_{j=1}^{\rank{F}} (\theta_s^F(h_s(Y, x_m)))_{ij} \tilde{f}'_j(Y, x_m).
	\end{equation}
If $(f_i)^*$ denotes the metric dual of $f_i$, and $(\tilde{f}_i)^*, (\tilde{f}'_i)^*$ are obtained by parallel transport of $(f_i)^*$ for the dual connection $\nabla^{F^*}$ on $F^*$ respectively along $\gamma, \delta$, then $(\tilde{f}_i)^*, (\tilde{f}'_i)^*$ are equal to the metric duals of $\tilde{f}_i$ and $\tilde{f}'_i$. This gives 
	\begin{equation}\label{matricepassagedual}
		(\tilde{f_i})^*(h_s(Y, x_m)) = \sum_{j=1}^{\rank{F}} (\overline{\theta_s^F}(h_s(Y, x_m)))_{ij} (\tilde{f}'_j)^*(Y, x_m),
	\end{equation}
therefore $\theta_s^{F^*}(h_s(Y', x_m'))_{ij} = \overline{\theta_s^{F}}(h_s(Y', x_m'))_{ij}$.
We have the following

\begin{lem}
	
	Let $G \in \cC^{\infty}(X \times X, F \otimes F^*)$, $s \in ]-\eta/2, \eta/2[$. Denote by $G_{x_0}(Z, Z')$ the localized kernel at $x_0 \in V_s$ obtained by parallel transport along $\gamma$ as above, and $G_{x_0}((Y, x_m), (Y', x_m'))$  the localized kernel in Fermi coordinates obtained by parallel transport along $\delta$ as above. For $x_m, x_m' \in ]-\eta-s, \eta-s[$ and $|Y|, |Y'| < \varepsilon$, we have

	\begin{equation}\label{componoyaux}
		G_{x_0}(h_s(Y, x_m), h_s(Y', x_m'))= (\theta_s^F)^T(h_s(Y, x_m))G_{x_0}((Y, x_m), (Y', x_m'))  \overline{\theta_s^F}(h_s(Y', x_m')).
	\end{equation}
\end{lem}

\begin{proof}

	Let us write 
		\begin{equation}\label{Gloc}
			G_{x_0}(Z, Z') = \sum_{i,j} c_{ij}(Z, Z') \tilde{f}_i(Z) \otimes \tilde{f}_j^*(Z'),
		\end{equation}
	By linearity of the parallel transport maps along $\gamma$ and $\delta$ for $\nabla^F$ we have
			
		\begin{equation}\label{Gprimeloc}
			G_{x_0}((Y, x_m), (Y', x_m')) = \sum_{i,j} c_{ij}(h_s(Y, x_m), h_s(Y', x_m')) \tilde{f}'_i(Y, x_m) \otimes (\tilde{f}'_j)^*(Y', x_m').
		\end{equation}
	From~\eqref{Gloc} we get
				
		\begin{equation}\label{eqabove}
			\begin{split}
				 G_{x_0}(h_s(Y, x_m), & h_s(Y', x_m')) \\
				 & = \sum_{i_1, j_1} c_{i_1j_1}(h_s(Y, x_m), h_s(Y', x_m')))\tilde{f}_{i_1}(h_s(Y, x_m)) \otimes \tilde{f}_{j_1}^*(h_s(Y', x_m')).
			\end{split}
		\end{equation}
	By~\eqref{matricepassage} and~\eqref{matricepassagedual}, the equation~\eqref{eqabove} becomes
				
		\begin{multline} 
			G_{x_0}(h_s(Y, x_m), h_s(Y', x_m')) = \sum_{\substack{i_1, j_1, k_1, k_2}} (\theta_s^F)^T(h_s(Y, x_m))_{k_1i_1} c_{i_1j_1}(h_s(Y, x_m), h_s(Y', x_m')) \\ \overline{\theta_s^F}(h_s(Y', x_m'))_{j_2 k_2} \tilde{f}'_{k_1}(Y, x_m) \otimes (\tilde{f}'_{k_2})^*(Y', x_m').
		\end{multline}
	By~\eqref{Gprimeloc}, this gives~\eqref{componoyaux}.
\end{proof}
By applying the above to $P_{p, x_0}$ for $F = L^p \otimes E$ we get 

	\begin{multline}\label{BergmanFermi}
		P_{p, x_0}((Y, x_m), (Y', x_m'))  \\
		=\left((\theta_s^{L^p \otimes E})^T\right)^{-1}(h_s(Y, x_m))P_{p, x_0}(h_s(Y, x_m), h_s(Y', x_m'))  \left(\overline{\theta_s^{L^p \otimes E}}\right)^{-1}(h_s(Y', x_m')).
	\end{multline}
Here and in the rest of the paper we will write $P_{p, x_0}((Y, x_m), (Y', x_m'))$ for the \textit{Bergman kernel expressed in the coordinates $((Y, x_m), (Y', x_m'))$}.

\subsection{Asymptotics of the diffeomorphism between Fermi and normal coordinates}\label{AsymptoticsDiffeoFermi}

In \cite[\S 2.2]{Finski24}, Finski proves an asymptotic expansion of order 2 for the diffeomorphism $h_s$ in the case of a complex submanifold of a complex manifold $X$. His method is inspired by an article of Gavrilov \cite{G07}.
In what follows we adapt his proof to our case and state the result in Theorem \ref{thasymptotics}. Consider a $k$-tensor $\alpha \in \cC^{\infty}(X, (T^*X)^{\otimes k})$. Recall that the operator $\nabla^{TX}$ acts on $\alpha$  by the formula, 
			
	\begin{equation}\label{tensorconnection}
		(\nabla^{TX} \alpha)(X_1, \ldots, X_{k+1}) = X_1 \alpha(X_2, \ldots, X_{k+1}) - \sum_{i=2}^{k+1} \alpha(X_2, \ldots, \nabla_{X_1}^{TX} X_i, \ldots, X_{k+1}),
	\end{equation}
with $X_1, \ldots X_{k+1} \in \cC^{\infty}(X, TX)$. This defines an operator $(\nabla^{TX})^{\otimes k} := \nabla^{TX} \circ \ldots \circ \nabla^{TX} :  \cC^{\infty}(X) \rightarrow \cC^{\infty}(X, (T^*X)^{\otimes k})$. As in \cite{Finski24}, the following lemma comes from \eqref{tensorconnection}.

	\begin{lem}[{\cite[Lemma 2.6]{Finski24}}, {\cite[\S 2]{G07}}]
		Let $\gamma(t)$ be a geodesic in $X$ and $v := \gamma'(t)$. Then for any $k \in \N$, the following identity holds
				
		\begin{equation}
			v^{\otimes k} \cdot (\nabla^{TX})^{\otimes k} = \left( \frac{\partial}{\partial v} \right)^k.
		\end{equation}
	\end{lem}

	\begin{cor}[{\cite[Corollary 2.7]{Finski24}}, {\cite[\S 2]{G07}}]

		Let $u \in \cC^{\infty}(X)$, then $u$ has the formal Taylor expansion at $x_0 \in X$,

		\begin{equation}\label{Tayloru}
			u(\exp^X_{x_0}(Z)) = \sum_{k=0}^{\infty} \frac{1}{k!} (Z^{\otimes k} \cdot (\nabla^{TX})^k u)(x_0),
		\end{equation}
		where $Z \in T_{x_0} X$.
	\end{cor}

From~\eqref{psiexph} and~\eqref{Tayloru}, for $x_0 \in V_s$ we have the expansion 
			
	\begin{equation}
		u(\psi_{x_0}(Y, x_m)) = \sum_{k=0}^{\infty} \frac{1}{k!} \left(h_s(Y, x_m)^{\otimes k} \cdot (\nabla^{TX})^k u \right)(x_0).
	\end{equation}
Let $y = \exp^{V_s}_{x_0}(Y)$, $|Y| < \vare'$. By Taylor expansion in $x_m$ around $x_m =0$ we have
			
	\begin{equation}\label{Taylorx_m}
		 u(\psi_{x_0}(Y, x_m)) = u \left(\exp^X_{y} \left(x_m e_{\kn}(y) \right) \right) = \displaystyle \sum_{k=0}^{\infty} \frac{1}{k!} \left( (x_m e_{\kn}(y))^{\otimes k} \cdot (\nabla^{TX})^k u \right)(\psi(Y, 0)).
	\end{equation}
Again as in \cite{Finski24}, expanding the terms of the sum in~\eqref{Taylorx_m} along the $Y$ component we get 
			
\begin{multline}\label{preTaylor2}
	\left( (x_m e_{\kn}(y))^{\otimes k}  \cdot (\nabla^{TX})^k u \right) (\psi(Y, 0))  \\
	= \sum_{l=0}^{\infty} \frac{1}{l!} Y^{\otimes l} \cdot (\nabla^{TV})^{\otimes l} \left( (x_m e_{\kn}(y))^{\otimes k} \cdot (\nabla^{TX})^{\otimes k} u \right)(x_0).
\end{multline}
The following lemma is a weaker version of \cite[Lemma 2.8]{Finski24}.
				
\begin{lem} 
	For $s \in ]-\eta/2, \eta/2[$, $x_0 \in V_s$, let $Y \in B^{T_{x_0}V_s}(0, \vare')$, let $y(t)$ be the geodesic $y(t) = \exp_{x_0}^{V_s}(tY)$, for $|t| <1$, and write $Y(t) = y'(t)$. Let $x_m \in ]-\eta-s, \eta-s[$. We have
	\begin{multline}\label{comuTaylor}
		Y^{\otimes l} \cdot (\nabla^{TV})^{\otimes l} \left( (x_m e_{\kn}(y))^{\otimes k} \cdot (\nabla^{TX})^{\otimes k} u \right)(x_0) \\
		 = \left( (Y^{\otimes l} \otimes (x_m e_{\kn}(x_0))^{\otimes k} ) \cdot (\nabla^{sp})^{\otimes l} (\nabla^{TX})^{\otimes k} u \right)(x_0).
	\end{multline}
\end{lem}
\begin{proof}
By the geodesic equation, and by construction of $\nabla^N$ and $\nabla^{sp}$, we have 

\begin{equation}\label{identitesconnexions}
	\nabla^{TV}_{Y}Y = 0, \quad \nabla_{Y}^{N} e_{\kn} = 0, \quad \nabla_{Y}^{sp} e_{\kn} = 0, \quad \textrm{and} \quad \nabla^{TX}_{e_{\kn}} e_{\kn}= 0.
\end{equation} 
As in \cite[Lemma 2.8]{Finski24}, \eqref{comuTaylor} is a consequence of \eqref{tensorconnection} and \eqref{identitesconnexions}.
\end{proof}
From \eqref{preTaylor2} and the above lemma, we get  another asymptotic formula for $u(\psi_{x_0}(Y, x_m))$ 
				
\begin{equation}\label{Tayloru2}
	\begin{split}
		u(\psi_{x_0}(Y, &x_m))		\\
		 &=\sum_{k=0}^{\infty} \sum_{l=0}^{\infty} \frac{1}{k!l!} \left( (Y^{\otimes l} \otimes (x_m e_{\kn}(x_0))^{\otimes k} ) \cdot (\nabla^{sp})^{\otimes l} (\nabla^{TX})^{\otimes k} u \right)(x_0).
	\end{split}
\end{equation}
We can decompose $h_s$ as
				
\begin{equation}\label{polydech}
	h_s(Y, x_m) = \sum_{r=1}^{\infty} h_s^{[r]}(Y, x_m),
\end{equation} 
where $h_s^{[r]}(Y, x_m)$ is a homogenous polynomial corresponding to the term of degree $r$ in $Y, x_m$ in the asymptotic expansion of $h_s$. By identifying the terms of~\eqref{Tayloru} and~\eqref{Tayloru2}, in the same way as~\cite{Finski24} we obtain the following result.

\begin{thm}\label{thasymptotics}
	The first two terms in the expansion of $h$ verify
				
		\begin{equation}\label{h1}
			h_s^{[1]}(Y, x_m) = Y + x_m e_{\kn}(x_0).
		\end{equation}
				
		\begin{equation}\label{h2}
			h_s^{[2]}(Y, x_m) = \frac{1}{2} A_{x_0}(Y)Y + x_mA_{x_0}(Y)(e_{\kn}(x_0)).
		\end{equation}
\end{thm}
		
\begin{proof}
	The proof for $h_s^{[1]}(Y, x_m)$ is straightforward. For $h_s^{[2]}(Y, x_m)$ we proceed as follows. Combining~\eqref{Tayloru} and~\eqref{Tayloru2} we get 

	\begin{multline}\label{interm}
		\left[\left(h_s^{[2]} (Y, x_m) \cdot \nabla^{TX} + \frac{1}{2}h_s^{[1]}(Y, x_m)^{\otimes 2} \cdot (\nabla^{TX})^{\otimes 2} \right)u\right](x_0) = \left[\left(\frac{1}{2} Y^{\otimes 2} \cdot (\nabla^{sp})^{\otimes 2} \right. \right. \\ 
		 \left. \left. + Y \otimes (x_m e_{\kn}(y)) \cdot \nabla^{sp} \otimes \nabla^{TX} +\frac{1}{2} (x_m e_{\kn}(y))^{\otimes 2} \cdot (\nabla^{TX})^{\otimes 2} \right) u\right](x_0).
	\end{multline}
By~\eqref{tensorconnection}, we have for $Y_1, Y_2 \in TV$,
		 
	\begin{equation}
		\begin{split}
			Y_1 \otimes Y_2 \cdot \nabla^{sp} \otimes \nabla^{TX} u &= Y_1(\nabla^{TX}_{Y_2} u) - \nabla^{TX}_{\nabla^{sp}_{Y_1} Y_2} u \\
		 	&= Y_1(Y_2u) - \nabla^{sp}_{Y_1} Y_2 u \\
		 	&= Y_1 \otimes Y_2 \cdot (\nabla^{TX})^{\otimes 2} u + A(Y_1)Y_2 u.
		\end{split}
	\end{equation}
The same computation gives
		 
	\begin{equation}
		Y_1 \otimes Y_2 \cdot (\nabla^{sp})^{\otimes 2} u = Y_1 \otimes Y_2 \cdot (\nabla^{TX})^{\otimes 2} u + A(Y_1)Y_2 u.
	\end{equation}
We obtain the equality~\eqref{h2} by~\eqref{h1} and~\eqref{interm}  .
\end{proof}

\subsection{Generalized Bergman kernel asymptotics in Fermi coordinates}\label{AsymptoticsGenBergmanFermi}
In this paragraph we combine the asymptotics of Theorems~\ref{ThBkernelasymptotics},~\ref{thasymptotics} to obtain the off-diagonal expansion of the \linebreak Bergman kernel in Fermi coordinates stated in Corollary~\ref{ThBKernelasymptoticsFermi}.

Recall that the map $\mJ_{x_0}$ was defined in \eqref{calJ0} by $\mJ_{x_0} = -2\pi \sqrt{-1} J_0(x_0)$, with $J_0$ the correspondence between the symplectic form and the Riemannian metric as in \eqref{J0}. In the rest of the paper, we will decompose $\mJ_{x_0}$ for $x_0 \in V_s$ as follows. Let $J_{x_0}$ denote the endomorphism induced by the complex structure $J$ acting on $T_{x_0}X$ and $H_{x_0}V_s := T_{x_0}V_s  \cap J_{x_0}T_{x_0}V_s$. Then $T_{x_0}X$ decomposes as the orthogonal direct sum
				 
\begin{equation}\label{splittangent}
	T_{x_0}X = H_{x_0}V_s \oplus \mathrm{Span}_{\R}(J_{x_0}e_{\kn}(x_0), e_{\kn}(x_0)),
\end{equation}
where $J_{x_0}e_{\kn}(x_0), e_{\kn}(x_0)$ are orthogonal and unitary. The space $H_{x_0}V$ is stable by $J$ hence by $J_0$, and we can write $J_0$ matrixwise as 
				
	\begin{equation}\label{splitJ0}					
		(J_0)_{x_0} = 	
		\begin{pmatrix}
			J_0 \vert_{H_{x_0}V_s} & 0 		& 0 	  \\
			0 			   & 0 		& -\alpha_{x_0}\\
			0			   & \alpha_{x_0} & 0
		\end{pmatrix},
	\end{equation}
with $\alpha_{x_0} := \omega(e_{\kn}, Je_{\kn})(x_0) > 0$, and we have 
					
		\begin{equation}\label{diagcalJ}					
			\mJ_{x_0} = 	
			\begin{pmatrix}
				\mJ_{x_0} \vert_{H_{x_0}V_s} & 0 & 0 			\\
				0 			   		 & 0 & 2i\pi \alpha_{x_0} \\
				0			   & -2i\pi \alpha_{x_0} & 0
			\end{pmatrix}, \ (\mJ_{x_0}^2)^{1/2} = 
			\begin{pmatrix}
				(\mJ_{x_0}^2)^{1/2} \vert_{H_{x_0}V_s} & 0 			 & 0			\\
				0 							   & 2\pi \alpha_{x_0} & 0 			\\
				0 							   & 0 			 & 2\pi \alpha_{x_0}
			\end{pmatrix}.
		\end{equation}
For $Y \in T_{x_0}V$ and $x_m \in \R$  we define 
\begin{equation}
\begin{gathered}
	z_m := x_{m-1} + ix_m \in \C, \quad \mathrm{where} \quad Y = Y_H + x_{m-1}Je_{\kn}(x_0),\\
	\mathrm{with} \quad  x_{m-1} \in \R, Y_H \in H_{x_0}V_s \quad  \mathrm{and} \quad |Y|^2 = |Y_H|^2 + x_{m-1}^2.
\end{gathered} 
\end{equation}
Let us introduce $\mP^H_{x_0}$ the horizontal kernel, defined for $Y_H, Y_H' \in H_{x_0}V$ by

		\begin{multline}\label{horizontalBkernel}
			\mP^H_{x_0}(Y_H, Y_H') := \frac{\det_{\C} \mJ_{x_0}\vert_{HV}}{(2\pi)^{n-1}} \exp \bigg( -\frac{1}{4} \scal{(\mJ_{x_0}\vert_{HV}^2)^{1/2}(Y_H-Y_H')}{(Y_H-Y_H')} \\
			+ \frac{1}{2}\scal{\mJ_{x_0}\vert_{HV} Y_H}{Y_H'}\bigg),
		\end{multline}
such that $\mP_{x_0}$ defined in \eqref{calP} verifies, by \eqref{diagcalJ},

		\begin{equation}\label{calPsplit}
			\mP_{x_0}\left(Y + x_m e_{\kn}(x_0), Y' + x_m' e_{\kn}(x_0)\right) = \mP^H_{x_0}(Y_H, Y_H')P_{\alpha_{x_0}}(z_m, z_m').
		\end{equation}
We recall that in the above $P_{\alpha_{x_0}}$ corresponds to the Bergman kernel on the complex plane as in Section \ref{SectionBargmann}. 

Let $dv_{TV}$ be the volume form on $(T_{x_0}V_s, g^{T_{x_0}V_s})$ and $\kappa^V_{x_0}$ the smooth positive function such that if $dv_{x_0}'$ denotes the form on $T_{x_0}X$ in Fermi coordinates obtained by the identification along the path~\eqref{defdelta}, we  have for $(Y, x_m) \in U_{\eta, \vare}(x_0)$
			
\begin{equation}
	dv_{x_0}'(Y, x_m) = \kappa_{V, x_0}(Y, x_m) dv_{TV}(Y)dx_m, \quad \mathrm{with} \quad \kappa_{V, x_0}(0, 0) = 1.
\end{equation}
As a Corollary of Theorem~\ref{ThBkernelasymptotics}, we get the following off-diagonal expansion for the Bergman kernel in Fermi coordinates. For $Z = (Y, x_m)$ and $\alpha = (\beta, j)$, we write $\frac{\partial^{|\alpha|}}{\partial Z^{\alpha}}$ for $\frac{\partial^{|\beta|}}{\partial Y^{\beta}} \frac{\partial^j}{\partial x_m^j}$.

\begin{cor}\label{ThBKernelasymptoticsFermi}
	Let $x_0 \in U_{\eta}$ and $s \in ]-\eta/2, \eta/2[$ such that $x_0 \in V_s$.
	There exist $\mathrm{End}(E)_{x_0}$-valued polynomials $J_{r, x_0}'$ in $Y, x_m, Y'$ and $x_m'$ such that for $|Y|, |Y'| < \vare$, $x_m, x_m' \in ]-\eta-s, \eta-s[$, we have the expansion
		\begin{multline}\label{BKernelasymptoticsFermi}
			\underset{|\alpha| + |\alpha'| \leq l}{\sup} \bigg| \frac{\partial^{|\alpha| + |\alpha'|}}{\partial Z^{\alpha} \partial Z'^{\alpha'}} \bigg(\frac{1}{p^n} P_{p, x_0}((Y, x_m), (Y', x_m')) \\
			-\sum_{r = 0}^k p^{-r/2}\mF'_{r, x_0}(p^{1/2}Y, p^{1/2}x_m, p^{1/2}Y', p^{1/2}x_m') 
			\kappa^{-1/2}_{V, x_0}(Y, x_m) \kappa^{-1/2}_{V, x_0}(Y', x_m')\bigg)\bigg|_{\cC^{l'}(X)} \\
			\leq Cp^{\frac{k-l+1}{2}}(1 + \sqrt{p}|(Y, x_m)| + \sqrt{p}|(Y', x_m')|)^M  \\
			\times \exp(-c \sqrt{\mu_0p}|(Y-Y', x_m-x_m')|)+ \mO(e^{-c_0 \sqrt{p}}),
		\end{multline}
with 

		\begin{equation}
			\mF'_{r, x_0}(Y, x_m, Y', x_m') := J_{r, x_0}'(Y, x_m, Y', x_m') \\
			\mP^H_{x_0}(Y_H, Y_H')P_{\alpha_{x_0}}(z_m, z_m'),
		\end{equation}
where $z_m = x_{m-1} + i x_m$, $Y = Y^H + x_{m-1}Je_{\kn}(x_0)$, $Y^H \in H_{x_0}V$ and the first term equals
			
		\begin{equation}\label{firsttermfermi}
			J_{0, x_0}' = \Id_{E_{x_0}}.
		\end{equation}
As in Theorem~\ref{ThBkernelasymptotics}, the above expansion holds uniformly  in $\cC^{l'}$-norm taken on $x \in \mU_{\eta}$ for $l' \in \N$.
\end{cor}

\begin{proof}
	The method for the proof is adapted from \cite[Proposition 5.4]{Finski24}. Let $x_0 \in V_s$, we have
		 	\begin{equation}\label{calPasympt}
				 \mP_{x_0}(p^{1/2}h_s(Y, x_m), p^{1/2}h_s(Y', x_m')) = \frac{\det_{\C} \mJ_{x_0}}{(2\pi)^{n}} \exp (-pf_s(Y, x_m, Y', x_m')).
			\end{equation}
The map $f_s$ is defined by 
		
			\begin{multline}
				f_s(Y, x_m, Y', x_m') = \frac{1}{4} \scal{(\mJ_{x_0}^2)^{1/2}(h_s(Y, x_m)-h_s(Y', x_m'))}{h_s(Y, x_m)-h_s(Y', x_m')} \\
				- \frac{1}{2} \scal{\mJ_{x_0}h_s(Y, x_m)}{h_s(Y',x_m')}.
			\end{multline}
By Theorem~\ref{thasymptotics} and the expansion~\eqref{polydech}, we write $f_s$ as 
		
			\begin{multline}
				f_s(Y, x_m, Y', x_m') = \frac{1}{4} \scal{(\mJ_{x_0}^2)^{1/2}(Y-Y')}{Y-Y'} + \frac{\pi \alpha_{x_0}}{2}(x_m'-x_m)^2 - \frac{1}{2} \scal{\mJ_{x_0}Y}{Y'} \\
				+ \frac{1}{2}x_m \scal{\mJ_{x_0}e_{\kn}(x_0)}{Y'}- \frac{1}{2}x_m' \scal{\mJ_{x_0}e_{\kn}(x_0)}{Y} + g_s(Y, x_m, Y', x_m'),
			\end{multline}
where $g_s$ is the remainder of the expansion of $f_s$ in $Y, Y', x_m, x_m'$ and is therefore of order at least $3$ in $Y, x_m, Y', x_m'$. Then by \eqref{calPsplit}, \eqref{calPasympt} becomes

	\begin{multline}
		\mP_{x_0}(p^{1/2}h_s(Y, x_m), p^{1/2}h_s(Y', x_m'))  \\
		=\mP_{x_0}^H(p^{1/2}Y_H, p^{1/2}Y_H')P_{\alpha_{x_0}}(p^{1/2}z_m, p^{1/2}z_m') \exp (-pg_s(Y, x_m, Y', x_m')).
	\end{multline}
We use the same trick as \cite[(7.2.19)]{MM07}, and write the Taylor expansion of $pg$ as follows
		
			\begin{equation}
				pg_s(Y, Y', x_m, x_m') = \sum_{i \geq 3} p^{(2-i)/2} g_s^{[i]}(p^{1/2}Y, p^{1/2}Y', p^{1/2}x_m, p^{1/2}x_m'),
			\end{equation}
with $g_s^{[i]}$ the term of order $i$ in the Taylor expansion of $g_s$. Combining with the exponential series expansion
			\begin{equation}
				\exp(-pg_s) = \sum_{k \geq 0} \frac{(-1)^k}{k!}(pg_s)^k, 
			\end{equation}
we get that there exist polynomials $Q_{i, x_0}$  such that
	
			\begin{multline}\label{devmPfermi}
				\mP_{x_0}(p^{1/2}h_s(Y, x_m), p^{1/2}h_s(Y', x_m'))\\
				=\sum_{i \geq 0} p^{n-i/2} Q_{i, x_0}(p^{1/2}Y, p^{1/2}x_m, p^{1/2}Y', p^{1/2}x_m')\mP^H_{x_0}(p^{1/2}Y_H, p^{1/2}Y_H')P_{\alpha_{x_0}}(p^{1/2}z_m, p^{1/2}z_m') \\
				+ \mO(p^{-\infty}).
			\end{multline}
In the above $Q_{0, x_0} = \Id_{E_{x_0}}$. By Theorem~\ref{ThBkernelasymptotics} we have 
			\begin{multline}\label{laststepdev}
				P_{p, x_0}(h_s(Y, x_m), h_s(Y', x_m'))  \\
				=\sum_{r \geq 0} p^{n-r/2} J_{r, x_0}(p^{1/2}h_s(Y, x_m), p^{1/2}h_s(Y', x_m')) \mP_{x_0}(p^{1/2}h_s(Y, x_m), p^{1/2}h_s(Y', x_m')) \\
				\kappa_{x_0}^{-1/2}(h_s(Y, x_m))\kappa_{x_0}^{-1/2}(h_s(Y', x_m')) + \mO(p^{-\infty}).
			\end{multline}
Again we write the expansion of $p^{1/2}h_s$ as 
			\begin{equation}
				p^{1/2}h_s(Y, x_m) = \sum_{i \geq 1} p^{(1-i)/2}h_s^{[i]}(p^{1/2}Y, p^{1/2}x_m), 
			\end{equation}
and we apply to~\eqref{laststepdev} the same method as before and write the asymptotics of $\kappa_{x_0}^{-1/2}(h_s(Y, x_m))$ and $\kappa_{x_0}^{-1/2}(h_s(Y', x_m'))$ in the same way as for $g_s$ and $h_s$ and combine with \eqref{devmPfermi}.

We then use~\eqref{BergmanFermi},  and proceed again with the asymptotics of $\theta_s^{L^p \otimes E}(Y, x_m)$ in the same way. After the above steps the asymptotics we obtain remain of the same form as~\eqref{devmPfermi}, only the polynomials change and we obtain~\eqref{BKernelasymptoticsFermi}. To compute the first term we use $Q_{0, x_0} = \Id_{E_{x_0}}$ and the first order expansions $\kappa_{x_0}(h_s(Y, x_m))\kappa_{V, x_0}^{-1}(Y, x_m)= 1 + p^{-1/2}\mO(|p^{1/2}Y, p^{1/2}x_m|)$ and $\theta_s^{L^p \otimes E}(Y, x_m) = \mathrm{Id} + p^{-1/2}\mO(|p^{1/2}Y, p^{1/2}x_m|)$ and get~\eqref{firsttermfermi}. 
	The fact that $h_s$ is a diffeomorphism ensures that $|h_s(Y, x_m)|$ is controlled by $|(Y, x_m)|$.
\end{proof}

\section{Asymptotics for the Toeplitz operator $T_{W,p}$}\label{SectionToeplitz}

This section is organized as follows. In Sect.~\ref{SectionKernToeplitz}, thanks to the appropriate choice of coordinates in~\eqref{splittangent},~\eqref{splitJ0},~\eqref{diagcalJ} and in the Bergman kernel expansion in Fermi coordinates~\eqref{BKernelasymptoticsFermi}, we obtain the full off-diagonal kernel expansion for $T_{W, p}$ by reducing the problem to the case of Section \ref{SectionBargmann}. This expansion is stated in Theorem~\ref{ThToeplitzindicasymptoticsFermi} and combined with the estimates of Lemma~\ref{Nonlocalizedasymptotics} it implies Theorem \ref{ThondiagonalTwexpansion}. In Sect.~\ref{PreuveCompoNoyau} we prove Theorem~\ref{thmNoyauCompo} by applying similar techniques. In Sect.~\ref{tracetoeplitz} we prove the trace asymptotics of Corollary~\ref{ThCEGeneralCase}, then we deduce Corollary~\ref{CorTraceToeplitz} from Corollary~\ref{ThCEGeneralCase}. In Sect.~\ref{Appendix} we prove the technical result stated in Lemma \ref{lemmeJq}. 

\subsection{Kernel asymptotics for $T_{W, p}$ and proof of Theorem \ref{ThondiagonalTwexpansion}}\label{SectionKernToeplitz}

For $x,y \in X$, by definition of $T_{W, p}$ we have

	\begin{equation}\label{toeplitznoyau}
		T_{W, p}(x, y) = \int_{W} P_p(x, z)P_p(z, y)dv_X(z).
	\end{equation}
For $s \in ]-\eta/2, \eta/2[$ and $x_0 \in V_s$, we introduce the following subsets of $U_{\eta, \vare}(x_0)$ and $\mU_{\eta, \vare}(x_0)$
				
		\begin{equation}
			U^-_{\eta, \vare}(x_0) = B^{T_{x_0}V_s}(0, \vare) \times ]-\eta-s, -s[, \qquad U^+_{\eta, \vare}(x_0) = B^{T_{x_0}V_s}(0, \vare) \times ]-s, \eta-s[,
		\end{equation}
				
		\begin{equation}
			\mU^-_{\eta, \vare}(x_0) = \psi_{x_0}(U^-_{\eta, \vare}(x_0)), \qquad \mU^+_{\eta, \vare}(x_0) = \psi_{x_0}(U^+_{\eta, \vare}(x_0)).
		\end{equation}
\begin{figure}[H]
	\begin{center}
	\begin{tikzpicture}[>=stealth]

	\fill[gray!20] (-1,2) rectangle (1,-2);
				
	\draw[-] (-3.5,0) -- (3,0) node [right] {$V= V_0$};

	\draw[dashed] (-3.5, 2) -- (3, 2) node [ right] {$\eta$};

	\draw[dashed] (-3.5, -2) -- (3, -2) node [ right] {$-\eta$};

	\draw[->] (-4, -2.5)--(-4, 2.5);

	\fill (-4, 1) circle (2pt) node [left] {$0$};

	\fill (-4, 0) circle (2pt) node [left] {$-s$};

	\fill (-4, -2) circle (2pt) node [left] {$-\eta-s$};

	\fill (-4, 2) circle (2pt) node [left] {$\eta-s$};

	\fill (0, 1) circle (2pt) node [below] {$x_0$};

	\draw (-3.5, 1)--(3, 1) node [right] {$V_s$};

	\draw[thick] (-1, 1)--(1, 1) node [midway, above] {\scriptsize $B^{V_s}(x_0, \vare)$};

	\fill[pattern=north east lines]
       (-1,0.95) rectangle (1,1.05);

	\draw [decorate, decoration={brace, amplitude=5pt}, thick] (-1.1,0.05) -- (-1.1,1.95) node [fill=white, midway, left=10pt] {$\mU_{\eta, \vare}^+(x_0)$};

	\draw [decorate, decoration={brace, amplitude=5pt}, thick] (-1.1,-1.95) -- (-1.1,-0.05) node [midway, left=10pt] {$\mU_{\eta, \vare}^-(x_0)$};

	\end{tikzpicture}
	\end{center}
	\caption{The neighborhoods $\mU_{\eta, \vare}^+(x_0)$ and $\mU_{\eta, \vare}^-(x_0)$, in the case $s \in [0, \eta/2].$}
	\label{Uplusmoins}
\end{figure}
Note that although the definition of $\mU_{\eta, \vare}^-(x_0)$ and $\mU_{\eta, \vare}^+(x_0)$ suggests they might depend on $s$, it is not the case (see figure \ref{Uplusmoins}). By \cite[Corollary 3.2]{Finski24}, the lemma we state below holds, as a consequence of the Bishop--Gromov inequality. 

		\begin{lem}[{\cite[Corollary 3.2]{Finski24}}]\label{ThmBishopG}
			Assume that there exists $s > 0$ such that the Ricci curvature of $(X, g^{TX})$ satisfies the lower bound 
			\begin{equation}
				\Ric_{g^{TX}} \geq -(m-1)s, 
			\end{equation} 
			where $m$ is the dimension of $X$. Then there exists a constant $C'>0$ which depends only on $n, s, r_X$ such that for any $x_0 \in X$, $k > 2(m-1)\sqrt{s}$, we have
			\begin{equation}\label{BishopG}
				\int_X \exp(-kd(x_0, x))dv_{X}(x) < \frac{C'}{k^m}.	
			\end{equation}
		
		\end{lem}
The following results come from the exponential estimates~\eqref{expestimates} combined with the inequality~\eqref{BishopG}.

\begin{lem}\label{Nonlocalizedasymptotics}
	Using the notations above there exist $p_0 \in \N$, $c > 0$, such that for $l \in \N$ there exist $C_{l, W}, C_{l, W}', C_{l, W}''> 0$ such that for $p \geq p_0$ :
			
	\begin{enumerate}[(i)]
	\item If $x, x' \in X$, 
			\begin{equation}
				|T_{W, p}(x,x')|_{\cC^l(X)} \leq C_{l, W}p^{n+l/2} e^{-c\sqrt{p}d(x,x')}.
			\end{equation}
	\item If $x \in W^c$,
			\begin{equation}
				|T_{W, p}(x, x)|_{\cC^l(X)} \leq C_{l, W}' p^{n+l/2}e^{-c\sqrt{p}d(x, V)}.
			\end{equation}
				
	\item If $x \in W \setminus V$, 
				
			\begin{equation}
				\left| T_{W, p}(x, x) - P_p(x, x) \right|_{\cC^l(X)} \leq C_{l, W}'' p^{n+l/2}e^{-c\sqrt{p}d(x, V)}.
			\end{equation}
					
	\item If $x \in \mU_{\eta, \vare}^+(x_0)$ and $y \in \mU_{\eta, \vare}(x_0)$ for some $x_0 \in V_s$, 
				
				\begin{equation}
					T_{W,p}(x,y) = P_p(x, y) - \int_{\mU_{\eta, \vare}^-(x_0)} P_p(x,z)P_p(z,y)dv_X(z) + \mO_{\cC^l(X)}(p^{n+l/2}e^{-c\eta\sqrt{p}}).
				\end{equation}
					
	\item If $x \in \mU_{\eta, \vare}^-(x_0)$ and $y \in \mU_{\eta, \vare}(x_0)$ for some $x_0 \in V_s$, 
			 	
			 	\begin{equation}
			 		T_{W,p}(x, y) = \int_{\mU_{\eta, \vare}^+(x_0)} P_p(x, z)P_p(z, y)dv_X(z) + \mO_{\cC^l(X)}(p^{n+l/2}e^{-c\eta\sqrt{p}}).
			 	\end{equation}

	\end{enumerate}
\end{lem}
	
	In (iv) and (v) the terms $\mO_{\cC^l}(p^{n+l/2}e^{-c\eta\sqrt{p}})$ are to be understood in the same sense as the inequalities in (i), (ii), (iii).
	\begin{proof}
		Since $(X, g^{TX})$ is of bounded geometry, there exists $s>0$ such that the Ricci curvature satisfies the lower bound 
		
					\begin{equation}
						\Ric_{g^{TX}} \geq -(2n-1)s,
					\end{equation} 
	therefore Lemma \ref{ThmBishopG} holds and we will make extensive use of it.
	\begin{enumerate}[(i)]
	\item If $x, x' \in X$, applying~\eqref{expestimates} in~\eqref{toeplitznoyau}, and the triangle inequality to get $d(x, x'') + d(x'', x') \geq \frac{1}{2} d(x, x') + \frac{1}{2} d(x, x'')$, by Lemma~\ref{ThmBishopG} we get that there exists $C_{l, W}$ such that 
		\begin{align}
			|T_{W, p}(x, x')|_{\cC^l(X)} 	& \leq C_{l, W} p^{n+l/2}e^{-c\sqrt{p}d(x, x')} \vol(W).
		\end{align}
	\item If $x \in W^c \setminus U_{\eta}$, applying Lemma~\ref{ThmBishopG} and \eqref{expestimates} in~\eqref{toeplitznoyau}:

		\begin{equation}
			|T_{W, p}(x, x)|_{\cC^l(X)} 	\leq C_l^2 p^{2n+l/2}e^{-c\sqrt{p}d(x, V)}\int_{W}  e^{-c\sqrt{p}d(x, z)} dv_X(z),
		\end{equation}
	we can then apply \eqref{BishopG} and get
		\begin{equation}
			|T_{W, p}(x, x)|_{\cC^l(X)} \leq C_{l, W}' p^{n+l/2}e^{-c\sqrt{p}d(x, V)} C_W,
		\end{equation}
		in particular $T_{W, p}(x, x)$ has exponential decay in $p$.

		\item  If $x \in W \setminus U_{\eta}$, by definition of $T_{f,p}$ and~\eqref{expestimates} we get
			\begin{equation}
			\begin{aligned}
				& \left| T_{W, p}(x, x) - P_p(x, x) \right|_{\cC^l(X)}  \leq \int_{W^c} |P_p(x, z)P_p(z, x)|_{\cC^l(X)} dv_X(z) \\
						& \leq C_l^2 p^{2n+l/2}e^{-c\sqrt{p}d(x, V)}\int_{ W^c}  e^{-c\sqrt{p}d(x, z)} dv_X(z) \\
						& \leq C_l^2 p^{2n+l/2}e^{-c\sqrt{p}d(x, V)}\int_{W^c}  e^{-c\sqrt{p}d(V, z)} dv_X(z) \\
						& \leq C_{l, W}'' p^{n+l/2}e^{-2c\sqrt{p}d(x, V)},
			\end{aligned}
			\end{equation}
			where for the last inequality we applied Lemma \ref{ThmBishopG} again.
			We conclude that $T_{W, p}(x, x)$ has the same asymptotics as the Bergman kernel up to exponentially decaying terms in $p$.

			\item If $x \in \mU_{\eta, \vare}^+(x_0)$ and $y \in \mU_{\eta, \vare}$, we split the integral expressing $T_{W, p}(x, y)$ as
			
					\begin{multline}
						T_{W, p}(x, y) = P_p(x,y) -\int_{\mU_{\eta, \vare}^-(x_0)} P_p(x, z)P_p(z, y) dv_X(z)  \\
				 		\qquad  - \int_{W^c \setminus \mU_{\eta, \vare}(x_0)} P_p(x, z)P_p(z, y) dv_X(z). 
					\end{multline}
					By~\eqref{expestimates} we have
					\begin{equation}
					\begin{aligned}
						\bigg|\int_{W^c \setminus \mU_{\eta, \vare}(x_0)} P_p(x, z)P_p(z, y) & dv_X(z) \bigg|_{\cC^l(X)} \leq \int_{W^c \setminus \mU_{\eta, \vare}(x_0)} |P_p(x, z)P_p(z, y)|_{\cC^l(X)} dv_X(z) \\
						& \leq C_l^2 p^{2n+l/2} e^{-c\eta\sqrt{p}} \int_{W^c \setminus \mU_{\eta, \vare}(x_0)} e^{-c\sqrt{p}d(z, y)}dv_X(z)\\
						& \leq C_l^2 p^{n+l/2} e^{-c\eta\sqrt{p}} C_W'.
					\end{aligned}
				\end{equation}
				As before, we obtain the last inequality by applying Lemma~\ref{ThmBishopG}. 

			\item If $x \in \mU_{\eta, \vare}^-(x_0)$, and $y \in \mU_{\eta, \vare}(x_0)$, we write 

					\begin{equation}
						T_{W, p}(x, y) = \int_{\mU_{\eta, \vare}^+(x_0)}P_p(x, z)P_p(z, y)dv_X(z) + \int_{W \setminus \mU_{\eta, \vare}(x_0)} P_p(x, z)P_p(z,y)dv_X(z), 
					\end{equation}
			applying again Lemma~\ref{ThmBishopG} and~\eqref{expestimates}, we obtain
					\begin{equation}
					\begin{aligned}
							\left|\int_{W \setminus \mU_{\eta, \vare}(x_0)} P_p(x, z)P_p(z, y) dv_X(z) \right|_{\cC^l(X)} & \leq \int_{W \setminus \mU_{\eta, \vare}(x_0)} |P_p(x, z)P_p(z, y)|_{\cC^l(X)} dv_X(z) \\
							& \leq C_l^2 C_W'p^{n+l/2} e^{-c\eta\sqrt{p}} .
					\end{aligned}
				\end{equation}
		\end{enumerate}
		This completes the proof of Lemma~\ref{Nonlocalizedasymptotics}.
	\end{proof}

\begin{thm}\label{ThToeplitzindicasymptoticsFermi}
	 Let $x \in U_{\eta}$ and $s \in ]-\eta/2, \eta/2[$ such that we can write $x = x_0 \in V_s$. With the notations of Corollary~\ref{ThBKernelasymptoticsFermi}, there exist $\mathrm{End}(E)_{x_0}$-valued polynomials $Q_{r, x_0}$ in $Y, x_m, Y'$ and $x_m'$, such that for $|Y|, |Y'| < \vare$, $x_m, x_m' \in ]-\eta-s, \eta-s[$, we have the expansion, 
		\begin{enumerate}[(i)]
			\item If $(Y, x_m) \in U_{\eta,\vare}^+(x_0)$,
				\begin{multline}\label{ToeplitzindicasymptoticsFermi+}
					T_{W, p, x_0}((Y, x_m), (Y', x_m')) \\
					=P_{p, x_0}((Y, x_m), (Y', x_m')) - P_{p, x_0}((Y, x_m), (Y', x_m'))\mathrm{erfc} \left(\sqrt{p\pi \alpha_{x_0}} \left(s + \frac{z_m-\overline{z}_m'}{2i}\right) \right) \\
					+ \sum_{r \geq 1} p^{n-r/2}\mQ_{r, x_0}(p^{1/2}Y, p^{1/2}x_m, p^{1/2}Y', p^{1/2}x_m', p^{1/2}s)\exp \left(-p\pi \alpha_{x_0}\left(s+\frac{z_m-\overline{z}_m'}{2i}\right)^2\right) \\
					\mspace{500mu}+ \mO(p^{-\infty}).
				\end{multline}
			\item If $(Y, x_m) \in U_{\eta, \vare}^-(x_0)$, 
				\begin{multline}\label{ToeplitzindicasymptoticsFermi-}
					T_{W, p, x_0}((Y, x_m), (Y', x_m'))\\
					=P_{p, 	x_0}((Y, x_m), (Y', x_m'))\mathrm{erfc} \left(\sqrt{p\pi \alpha_{x_0}} \left(-s - \frac{z_m-\overline{z}_m'}{2i}\right) \right) \\
					+ \sum_{r \geq 1} p^{n-r/2}\mQ_{r, x_0}(p^{1/2}Y, p^{1/2}x_m, p^{1/2}Y', p^{1/2}x_m', p^{1/2}s)\exp \left(-p\pi \alpha_{x_0}\left(s+\frac{z_m-\overline{z}_m'}{2i}\right)^2\right) \\
					\mspace{500mu}+ \mO(p^{-\infty}).
				\end{multline}
		\end{enumerate}
In the above the quantity $\mQ_{r, x_0}(Y, x_m, Y', x_m')$ is defined by

		\begin{equation}
			\mQ_{r, x_0}(Y, x_m, Y', x_m') := Q_{r, x_0}(Y, x_m, Y', x_m') \mP_{x_0}(Y_H, Y_H')P_{\alpha_{x_0}}(z_m, z_m'), 
		\end{equation}
	and $\alpha_{x_0} := \omega_{x_0}(e_{\kn}(x_0), J_{x_0}e_{\kn}(x_0))$. Notice that since from~\eqref{identiteserf} we have $\mathrm{erfc}(z) = 1- \mathrm{erfc}(-z)$, the formulas~\eqref{ToeplitzindicasymptoticsFermi+} and~\eqref{ToeplitzindicasymptoticsFermi-} are actually the same.
\end{thm}

\begin{proof}
	We start by proving (ii). By (v) of Lemma~\ref{Nonlocalizedasymptotics} we have, for $(Y, x_m)\in U_{\eta, \vare}^-(x_0)$ and $(Y', x_m') \in U_{\eta, \vare}(x_0)$,
		\begin{multline}
			T_{W,p,x_0}((Y, x_m), (Y', x_m')) = \int_{(Y'', x_m'')\in U_{\eta, \vare}^+} P_{p, x_0}((Y, x_m),(Y'', x_m'')) P_{p, x_0}((Y'', x_m''),(Y',x_m'))\\
			\kappa_{V, x_0}(Y'',x_m'') dv_{TY}(Y'')dx_m'' + \mO(p^{-\infty}).
		\end{multline}
	By the Bergman kernel asymptotics~\eqref{BKernelasymptoticsFermi}, $T_{W,p,x_0}((Y, x_m), (Y', x_m'))$ admits the expansion
		\begin{multline}
			T_{W, p, x_0}((Y, x_m), (Y', x_m')) = \sum_{r, r' \geq 0} p^{2n - (r + r')/2} \kappa_{V, x_0}^{-1/2}(Y, x_m) \kappa_{V, x_0}^{-1/2}(Y', x_m') \\
			I_{r, r', x_0}(p^{1/2}Y, p^{1/2}x_m, p^{1/2}Y', p^{1/2}x_m'),
		\end{multline}
	where
		\begin{multline}
			I_{r, r', x_0}(p^{1/2}Y, p^{1/2}x_m, p^{1/2}Y', p^{1/2}x_m') := \int_{(Y'', x_m'') \in U_{\eta, \vare}^+}\mF'_{r, x_0}(p^{1/2}Y, p^{1/2}x_m, p^{1/2}Y'', p^{1/2}x_m'') \\
			\mF'_{r', x_0}(p^{1/2}Y'', p^{1/2}x_m'', p^{1/2}Y', p^{1/2}x_m')dv_{TY}(Y'')dx_m''.
		\end{multline}
	By the exponential decay property of the Bergman kernel, up to a $\mO(p^{-\infty})$ term the integral above is the same as the integral over $Y'' \in T_{x_0}V_s$ and $-s<x_m''< + \infty$, which corresponds to integrating over $Y_H'' \in H_{x_0}V_s$ and $z_m'' \in \mathbb{H}_{-s}$. After a change of variable in $p^{1/2}Y_H'', p^{1/2}z_m''$, we apply Corollary~\ref{ToeplitzH} where we replace $\alpha$ by $\alpha_{x_0}$ and $s$, $z$, $z'$ by $-p^{1/2}s$, $p^{1/2}z_m$, $p^{1/2}z_m'$: there exist $Q_{r, r', x_0}$, $R_{r, r', x_0}$ polynomials such that 
		\begin{multline}
			I_{r, r', x_0}(p^{1/2}Y, p^{1/2}x_m, p^{1/2}Y', p^{1/2}x_m') \\
			= p^{-n}R_{r, r', x_0}(p^{1/2}Y_H, p^{1/2}z_m, p^{1/2}Y'_H, p^{1/2}z_m')\mathrm{erfc} \left(\sqrt{p\pi \alpha_{x_0}}\left(-s - \frac{z_m-\overline{z}_m'}{2i}\right) \right) \\
			+ p^{-n}Q_{r, r', x_0}(p^{1/2}Y_H, p^{1/2}z_m, p^{1/2}Y'_H, p^{1/2}z_m', p^{1/2}s)\exp \left(-p\pi \alpha_{x_0}\left(-s-\frac{z_m-\overline{z}_m'}{2i}\right)^2\right).
		\end{multline}
	By Corollary~\ref{ToeplitzH}, the terms $R_{r, r', x_0}$ correspond exactly to the terms in the asymptotics of the Bergman kernel and $Q_{0, 0, x_0} = 0$. A similar computation starting from (iv) of Lemma~\ref{Nonlocalizedasymptotics} proves (i).
\end{proof}

\subsubsection*{\textbf{Proof of Theorem \ref{ThondiagonalTwexpansion}}} Theorem~\ref{ThondiagonalTwexpansion} is a particular case of the results in Lemma~\ref{Nonlocalizedasymptotics} and Theorem~\ref{ThToeplitzindicasymptoticsFermi}. \qed

\subsection{Proof of Theorem \ref{thmNoyauCompo}}\label{PreuveCompoNoyau} This paragraph contains the proof of Theorem~\ref{thmNoyauCompo}. 

\subsubsection*{\textbf{Proof of (\ref{One})}} We start by proving (\ref{One}) of Theorem \ref{thmNoyauCompo}. We begin with the case $g(X) = X^q-X^{q+1}$ and we introduce
	\begin{equation}\label{definitionPI}
		\Pi_{p, q}(x_1, \ldots, x_q) := P_p(x_1, x_2)\ldots P_p(x_{q-1}, x_{q})P_p(x_{q}, x_1).
	\end{equation}
We will use the same formula as \cite[Lemma 5.1]{CE20},
	\begin{equation}\label{TqminusTqp1}
		(T_{W, p}^{q}-T_{W, p}^{q+1})(x, x) = \int_{W^q \times W^c} \Pi_{p, q+2}(x, x_1 \ldots, x_{q+1}) dv_X(x_1)\ldots dv_X(x_q).
	\end{equation}
Let $x \in X$, $x_1, \ldots, x_{q} \in W$ and $x_{q+1} \in W^c$. 
By the triangle inequality and the exponential estimates \eqref{expestimates} we have, on the one hand
\begin{equation}
		|\Pi_{p, q+2}(x, x_1 \ldots, x_{q+1})|_{\cC^k(X^{q+2})} \leq C_k^{q+2}p^{(q+2)n+k/2}e^{-c \sqrt{p} \max_{1 \leq i \leq q+1} d(x, x_i)} e^{-c\sqrt{p}d(x, x_{q+1})},
\end{equation}
and on the other hand
\begin{equation}
		|\Pi_{p, q+2}(x, x_1 \ldots, x_{q+1})|_{\cC^k(X^{q+2})} \leq C_k^{q+2}p^{(q+2)n+k/2}e^{-c \sqrt{p} \max_{1 \leq i \leq q+1} d(x, x_i)} e^{-c\sqrt{p}d(x, x_1)}.
\end{equation}
 If $x \in W^c, \ d(x, x_1) > d(x, V)$ and if $x \in W \setminus V, \ d(x, x_{q+1}) > d(x, V)$. From the above inequalities we get
\begin{equation}\label{InegalitePI}
		|\Pi_{p, q+2}(x, x_1 \ldots, x_{q+1})|_{\cC^k(X^{q+2})} \leq C_k^{q+2}p^{(q+2)n+k/2}e^{-c \sqrt{p} \max_{1 \leq i \leq q+1} d(x, x_i)} e^{-c\sqrt{p}d(x, V)}.
\end{equation}

Applying \eqref{InegalitePI} and the inequality $\max_{1 \leq i \leq q+1} d(x_0, x_i) \geq \frac{1}{q+1}(d(x_0, x_1) + d(x_0, x_2) + \ldots + d(x_0, x_{q+1}))$ we obtain 
\begin{equation}
	\begin{split}
		&\bigg|(T_{W, p}^{q}-T_{W, p}^{q+1})(x, x)\bigg|_{\cC^k(X)} \\
		&\leq C_k^{q+2}p^{(q+2)n+k/2}e^{-c\sqrt{p}d(x, V)}\left(\int_{W}  e^{- \frac{c\sqrt{p}}{q+1}d(x, x_1)}dv_X(x_1)\right)^q \left( \int_{W^c}e^{-\frac{c\sqrt{p}}{q+1}d(x, x_{q+1})}dv_X(x_{q+1})\right).
	\end{split}
\end{equation}
By the Bishop--Gromov inequality \eqref{BishopG} we get that there exists $C_{k, q, n} > 0$ such that
\begin{equation}
	\bigg|(T_{W, p}^{q}-T_{W, p}^{q+1})(x, x)\bigg|_{\cC^k(X)} \leq C_{k, q, n} p^{n+k/2}e^{-c\sqrt{p}d(x, V)}.
\end{equation}

Point (\ref{One}) of Theorem \ref{thmNoyauCompo} is a consequence of the above inequality by writing any polynomial vanishing at $0$ and $1$ as a linear combination of $\{X^q-X^{q+1}\}_{q \geq 1}$. \qed

\subsubsection*{\textbf{Proof of (\ref{Two})}} Let us now prove point (\ref{Two}) of Theorem \ref{thmNoyauCompo}. 

\subsubsection*{Proof of \eqref{DevOnDiagCompo}} Let $q \in \N \setminus \{0 \}$. Recall the quantity $\Pi_{p, q+2}$ was defined in \eqref{definitionPI}. To compute the kernel of the operator $T_{W, p}^q-T_{W, p}^{q+1}$ we will use again the formula \eqref{TqminusTqp1}. By the triangle inequality and the exponential estimates~\eqref{expestimates},  we get

	\begin{equation}\label{decayPI}
		|\Pi_{p, q+2}(x, x_1 \ldots, x_{q+1})|_{\cC^k(X^{q+2})} \leq C_k^{q+2}p^{(q+2)n+k/2}e^{-c \sqrt{p} \max_{1 \leq i \leq q+1} d(x, x_i)} e^{-c\sqrt{p}d(x_q, x_{q+1})},
	\end{equation}
which we can write as, for $x_1, \ldots, x_q \in W$, $x_{q+1} \in W^c$,

	\begin{equation}
		|\Pi_{p, q+2}(x, x_1, \ldots, x_{q+1})|_{\cC^k(X^{q+2})} \leq C_k^{q+2}p^{(q+2)n+k/2}e^{-c \sqrt{p} \max_{1 \leq i \leq q+1} d(x, x_i)} e^{-c\sqrt{p}d(x_{q+1}, W)}.
	\end{equation}
Let $s \in ]-\eta/2, \eta/2[$ and $x_0 \in V_s$ such that $x_0 = \exp^X_{y_0}(s e_{\kn}(y_0))$ for some $y_0 \in V$. By \eqref{TqminusTqp1} we have 
	\begin{equation}
	\begin{split}
		&\bigg|(T_{W, p}^{q}-T_{W, p}^{q+1})(x_0, x_0) \\
		& \qquad \qquad - \int_{(\mU_{\eta, \vare}^+(x_0))^q \times \mU_{\eta, \vare}^-(x_0)} \Pi_{p, q+2}(x_0 , x_1 \ldots, x_{q+1}) dv_X(x_1)\ldots dv_X(x_{q+1}) \bigg|_{\cC^k(X)} \\
		& \leq \int_{(W \setminus\mU_{\eta, \vare}^+(x_0))^q \times (W^c \setminus \mU_{\eta, \vare}^-(x_0))} C_k^{q+2}p^{(q+2)n+k/2} e^{-c \sqrt{p} \max_{1 \leq i \leq q+1} d(x_0, x_i)} \\
		& \mspace{300mu} \times e^{-c\sqrt{p}d(x_{q+1}, W)}dv_X(x_1)\ldots dv_X(x_{q+1}).
	\end{split}
	\end{equation}
Since $\max_{1 \leq i \leq q+1} d(x_0, x_i) \geq \frac{1}{q+1}(d(x_0, x_1) + d(x_0, x_2) + \ldots + d(x_0, x_{q+1}))$ and $d(x_{q+1}, W) \geq \eta$ we get
	\begin{equation}
	\begin{split}
		&\bigg|(T_{W, p}^{q}-T_{W, p}^{q+1})(x_0, x_0) \\
		& \qquad \qquad - \int_{(\mU_{\eta, \vare}^+(x_0))^q \times \mU_{\eta, \vare}^-(x_0)} \Pi_{p, q+2}(x_0 , x_1 \ldots, x_{q+1}) dv_X(x_1)\ldots dv_X(x_{q+1}) \bigg|_{\cC^k(X)} \\
		& \leq C_k^{q+2}p^{(q+2)n+k/2}e^{-c\sqrt{p}\eta}\left(\int_{(W \setminus\mU_{\eta, \vare}^+(x_0))}  e^{- \frac{c\sqrt{p}}{q+1}d(x_0, x_1)}dv_X(x_1)\right)^q \\
		&\mspace{300mu}\times \left( \int_{W^c \setminus \mU_{\eta, \vare}^-(x_0)}e^{-\frac{c\sqrt{p}}{q+1}d(x_0, x_{q+1})}dv_X(x_{q+1})\right).
	\end{split}
\end{equation}
Applying Lemma \ref{ThmBishopG} to the integral terms in the above inequality, we obtain
\begin{multline}
		(T_{W, p}^{q}-T_{W, p}^{q+1})(x_0, x_0) = \int_{(\mU_{\eta, \vare}^+(x_0))^q \times \mU_{\eta, \vare}^-(x_0)} \Pi_{p, q+2}(x_0 , x_1 \ldots, x_{q+1}) dv_X(x_1)\ldots dv_X(x_q) \\
		+ \mO(p^{n+k/2}e^{-c\sqrt{p}\eta}).
\end{multline}
The localized version of the above equation becomes 
	\begin{multline}
		(T_{W, p, x_0}^{q}-T_{W, p, x_0}^{q+1})(0, 0)=\int_{(U_{\eta, \vare}^+(x_0))^q \times U_{\eta, \vare}^-(x_0)} \Pi_{p, q+2, x_0}(0, (Y_1, x_{m, 1}), \ldots, (Y_{q+1}, x_{m, q+1}))\\
		\times\kappa_{V, x_0}(Y_1, x_{m, 1}) \ldots \kappa_{V, x_0}(Y_{q+1}, x_{m, q+1}) dv_{TY}(Y_1) \ldots dv_{TY}(Y_{q+1}) dx_{m, 1}\ldots dx_{m, q+1} + \mO(p^{-\infty}),
	\end{multline}
where 
	\begin{multline}
		\Pi_{p, q+2, x_0}(0, (Y_1, x_{m, 1}), \ldots, (Y_{q+1}, x_{m, q+1})) := \\
		P_{p, x_0}(0, (Y_1, x_{m, 1})) \ldots P_{p, x_0}((Y_{q}, x_{m, q}), (Y_{q+1}, x_{m, q+1})) P_{p, x_0}((Y_{q+1}, x_{m, q+1}), 0).
	\end{multline}
By the same procedure as in the proof of Theorem~\ref{ThToeplitzindicasymptoticsFermi}, we  decompose the above integral in the coordinates $Y_{i, H}, z_{m, i}$, use the Taylor expansion of $h_s(Y, x_m)$ and of $\kappa_{V, x_0}$, and integrate over $W_{H, i} = p^{1/2}Y_{H, i} \in T_{x_0}H$. We then obtain

	\begin{equation}\label{asymptdiagxs}
		(T_{W, p}^{q}-T_{W, p}^{q+1})(x_0, x_0) = \sum_{k \geq 0} p^{n-k/2} I_{k, q, x_0}(p^{1/2}s) + \mO(p^{-\infty}),
	\end{equation}
where there exist $\mathrm{\End}(E)_{x_0}$-valued polynomials $a_{k, q, x_0}$ such that after a change of variable $z_i = p^{1/2}z_{m, i}$ the integral over $z_{m, i}$ writes as

	\begin{multline}\label{ExpressionIq}
		I_{k, q, x_0}(p^{1/2}s) =  \int_{\substack{\Im(z_1), \ldots, \Im(z_q) > -p^{1/2}s \\ \Im(z_{q+1})< -p^{1/2}s}} a_{k, q, x_0}(z_1, \ldots , z_{q+1})\\ 
		\times\exp \bigg(- \frac{\pi \alpha_{x_0}}{2} \scal{MZ}{Z}\bigg)d \lambda(z_1)\ldots d\lambda(z_{q+1}),
	\end{multline}
with $\lambda$ the Lebesgue measure and $M$ such that

	\begin{equation}
		-\pi \alpha_{x_0} z_1(\overline{z}_1-\overline{z}_2)- \ldots \\
		- \pi \alpha_{x_0} z_q(\overline{z}_q-\overline{z}_{q+1}) -\pi \alpha_{x_0} |z_{q+1}|^2 = - \frac{\pi \alpha_{x_0}}{2} \scal{MZ}{Z},
	\end{equation}
so that we have 
	\begin{equation}
		P_{\alpha_{x_0}}(0, z_1)P_{\alpha_{x_0}}(z_1, z_2)\ldots P_{\alpha_{x_0}}(z_{q+1}, 0) = \alpha_{x_0}^{q+2}\exp \bigg(- \frac{\pi \alpha_{x_0}}{2} \scal{MZ}{Z}\bigg),
	\end{equation}
with $Z = (x_1, y_1, \ldots, x_{q+1}, y_{q+1})$, $z_j = x_j + iy_j$. From the first term in the expansion of the Bergman kernel \eqref{BKernelasymptoticsFermi}, in the formula \eqref{ExpressionIq} notice that we get 
	\begin{equation}\label{Expressiona_0}
		a_{0, q, x_0}(z_1, \ldots , z_{q+1}) = \frac{{\det}_{\C} (\mJ_{x_0} \vert_{HV})}{(2\pi)^{n-1}} \alpha_{x_0}^{q+2},
	\end{equation}
the term $(2\pi)^{-(n-1)}\det_{\C} \mJ_{x_0}\vert_{HV}$ is what remains after integration along the variables $Y_{i, H}$ which amounts to composing $q+1$ times the horizontal Bergman kernel $\mP_{x_0}^H$ defined in~\eqref{horizontalBkernel}. The  matrix $M$ in \eqref{ExpressionIq} can be given explicitly as the $2(q+1) \times 2(q+1)$-matrix defined blockwise by $M_{j, j} = 2 \Id_{2}, 1 \leq j \leq q+1$ and 
	\begin{equation}\label{DefinitionM}
		M_{j, j+1} = \begin{pmatrix} -1 & i \\ -i & -1 \end{pmatrix}, \quad M_{j+1, j} = (M_{j, j+1})^t, \quad 1 \leq j \leq q,
	\end{equation}
and $M_{i, j} = 0$ otherwise. Notice that $\Re(M)$ is positive definite. Let $\mu > 0$ be the smallest eigenvalue of $\Re(M)$. Let $d_k$ be the degree of $a_{k, q, x_0}$. On the compact tubular neighborhood $\mU_{\eta}$ the coefficients $\alpha_{x_0}$ and the coefficients in front of each monomial of the polynomials $a_{k, q, x_0}$ are bounded and we have an inequality of the form, for $\alpha = \min_{x_0 \in U_{\eta}} \alpha_{x_0}$,

	\begin{equation}
		|I_{k, q, x_0}(p^{1/2}s)| \leq  \int_{\substack{\Im(z_1), \ldots, \Im(z_q) > -p^{1/2}s \\ \Im(z_{q+1})< -p^{1/2}s}} C_{k, q}(1 + |Z|^2)^{d_k/2} \exp \left( -\frac{\pi \alpha}{2}\mu |Z|^2 \right)d \lambda(z_1)\ldots d\lambda(z_{q+1}),
	\end{equation}
which implies by integrating over $z_1, \ldots z_q \in \C,$

	\begin{equation}\label{rightexpdecay}
		|I_{k, q, x_0}(p^{1/2}s)| \leq  \int_{\substack{ \Im(z_{q+1})< -p^{1/2}s}} C'_{k, q}(1 + |z_{q+1}|^2)^{d_k/2} \exp \left( -\frac{\pi \alpha}{2}\mu |z_{q+1}|^2 \right) d\lambda(z_{q+1}).
	\end{equation}
therefore $I_{k, q, x_0}(p^{1/2}s)$ has exponential decay in $p^{d_k/2}s^{d_k}\exp\left(-\frac{\mu \pi \alpha}{2} ps^2\right)$ for $s > 0$ Note that the above inequality only depends on $s > 0$ and not on $x_0\in V_s$. To handle the case $s < 0$, we instead write 

	\begin{multline}\label{leftexpdecay}
		|I_{k, q, x_0}(p^{1/2}s)| \leq  \int_{\substack{\Im(z_1), \ldots, \Im(z_q) > -p^{1/2}s}} C''_{k, q}(1 + |z_1|^2 + \ldots +|z_q|^2)^{d_k/2} \\
		\exp \left( -\frac{\pi \alpha}{2}\mu (|z_1|^2 + \ldots +|z_q|^2)\right)d \lambda(z_1)\ldots d\lambda(z_q).
	\end{multline}
This proves \eqref{DevOnDiagCompo} for $g(X) = X^q-X^{q+1}$. When $g$ is any polynomial vanishing at $0$ and $1$, we conclude by decomposing $g$ in the basis $\{X^q-X^{q+1}\}_{q \geq 1}$. 
\subsubsection*{Proof of \eqref{FirstTermDevOnDiagCompo}} We first state a technical lemma, which we prove later on in Section \ref{Appendix}.

\begin{lem}\label{lemmeJq}
Let $M$ be the matrix defined in \eqref{DefinitionM}. Set
\begin{equation}\label{definitionJq}
	J_q(s):= \int_{\substack{\Im(z_1), \ldots, \Im(z_q) > -s \\ \Im(z_{q+1}) < -s}}\exp \left(- \frac{\pi}{2} \scal{MZ}{Z}\right) d\lambda(z_1) \ldots d\lambda(z_{q+1}),
\end{equation}
Then we have 
\begin{equation}\label{egaliteJq}
J_q(s) =  \int_{-\infty}^{\infty} \big(\mathrm{erfc}^{q}(y)-\mathrm{erfc}^{q+1}(y)\big)\frac{1}{\sqrt{\pi}}e^{-(y+\sqrt{2\pi}s)^2}dy.
\end{equation}
\end{lem}
From \eqref{Expressiona_0}, applying the change of variable $Z' = \sqrt{\alpha_{x_0}}Z$ in \eqref{ExpressionIq}, we get
\begin{equation}\label{exprI0}
	I_{0, q, x_0}(p^{1/2}s) = \frac{{\det}_{\C} (\mJ_{x_0} \vert_{HV})}{(2\pi)^{n-1}} \alpha_{x_0}J_q(s\sqrt{p\alpha_{x_0}}),
\end{equation}
by \eqref{egaliteJq}, this proves \eqref{FirstTermDevOnDiagCompo} for $g(X) = X^q-X^{q+1}$. When $g$ is any polynomial vanishing at $0$ and $1$, we conclude by decomposing $g$ in the basis $\{X^q-X^{q+1}\}_{q \geq 1}$. This proves (\ref{Two})  of Theorem \ref{thmNoyauCompo} and finishes the proof of the theorem.\qed

\subsection{Trace asymptotics for $T_{W, p}$, proofs of Corollary~\ref{ThCEGeneralCase} and Corollary~\ref{CorTraceToeplitz}}\label{tracetoeplitz}

In this paragraph we prove Corollary~\ref{ThCEGeneralCase}, from which we then deduce Corollary~\ref{CorTraceToeplitz}. The proofs also start with the case $g(X) = X^q-X^{q+1}$.
\subsubsection*{\textbf{Proof of Corollary~\ref{ThCEGeneralCase}}} 
We integrate~\eqref{asymptdiagxs} over $s \in ]-\eta/2, \eta/2[$ and obtain
	\begin{align}
	\begin{split}\label{lastlinetraceq}
		\tr(T_{W, p}^q - & T_{W, p}^{q+1}) = \int_{-\eta/2}^{\eta/2}ds\int_{x_0 \in V_s} \tr\left((T_{W, p}^{q}-T_{W, p}^{q+1})(x_0, x_0)\right) \kappa(x_0, s)dv_V(x_0) + \mO(p^{-\infty})\\
		&= \sum_{k \geq 0} p^{n-k/2} \int_{-\eta/2}^{ \eta/2}ds\int_{x_0 \in V_s} \tr(I_{k, q, x_0}(p^{1/2}s)) \kappa(x_0, s)dv_V(x_0) + \mO(p^{-\infty})\\
		&= \sum_{k \geq 0} p^{n-k/2} \int_{-\eta/2}^{ \eta/2} J_{k, q, s}(p^{1/2}s)ds + \mO(p^{-\infty}),
	\end{split}
\end{align}
where $\kappa$ is such that $dv_X(x_0) = \kappa(x_0, s)dv_V(x_0)ds$ with $\kappa(x_0, 0) = 1$ and
to get the form~\eqref{lastlinetraceq} we take the trace and write the Taylor expansion of $\kappa(x_0, s)$ as

	\begin{equation}\label{cheatcodeTaylor}
		\kappa(x_0, s) = \sum_{k \geq 0} p^{-k/2}c_{k, x_0}(p^{1/2}s)^k + \mO(p^{-\infty}),
	\end{equation}
where $c_{0, x_0} = 1$, and then we integrate over $x_0 \in V_s$. By change of variable $s' = p^{1/2}s$,~\eqref{lastlinetraceq} becomes  
	\begin{equation}
		\tr(T_{W, p}^q - T_{W, p}^{q+1}) = \sum_{k \geq 0} p^{n-k/2-1/2} \int_{-\sqrt{p}\eta/2}^{\sqrt{p}\eta/2} J_{k, q, p^{-1/2}s}(s)ds + \mO(p^{-\infty}).
	\end{equation}
Decouple the variables and write the expansion of $J_{k, q, t}(s)$ at $t=0$ uniformly in $s$ as
	\begin{equation}\label{Jexpansion}
		J_{k, q, t}(s) = \sum_{j \geq 0} f_{k,j,q}(s)t^j,
	\end{equation}
where each $f_{k, j}$ has exponential-type decay in $s^2$ by~\eqref{rightexpdecay} and~\eqref{leftexpdecay}. We now get 
	\begin{equation}
		\tr(T_{W, p}^q - T_{W, p}^{q+1}) = \sum_{j, k \geq 0} p^{n-(k+j)/2-1/2} \int_{-\sqrt{p}\eta/2}^{\sqrt{p}\eta/2} f_{k, j,q}(s)s^jds + \mO(p^{-\infty}).
	\end{equation}
By the decay property of the functions $f_{k,j,q}$, this gives
	\begin{equation}
		\tr(T_{W, p}^q - T_{W, p}^{q+1}) = \sum_{j, k \geq 0} p^{n-(k+j)/2-1/2} \int_{-\infty}^{ +\infty} f_{k, j,q}(s)s^jds + \mO(p^{-\infty}).
	\end{equation}
We therefore obtain the asymptotics of the trace
\begin{equation}\label{devTraceTWp}
	\tr(T_{W, p}^q - T_{W, p}^{q+1}) = \sum_{k \geq 1} p^{n-k/2}a_{k, q}(W) + \mO(p^{-\infty}).
\end{equation}
Let us now compute the coefficient $a_{1, q}(W)$. The integral we need to compute involves the first coefficients of each expansion~\eqref{lastlinetraceq},~\eqref{cheatcodeTaylor},~\eqref{Jexpansion}, therefore 

	\begin{equation}
		a_{1,q}(W) = \rank(E)\int_{-\infty}^{+\infty}\int_{V} I_{0, q, x_0}(s)dv_V(x_0)ds
	\end{equation}
By \eqref{exprI0} and Fubini we have

	\begin{equation}\label{expressiona1W}
		a_{1,q}(W) = \rank(E)\int_{V} \frac{{\det}_{\C} (\mJ_{x_0} \vert_{HV})}{(2\pi)^{n-1}} \alpha_{x_0}\left[\int_{-\infty}^{\infty} J_q(s\sqrt{\alpha_{x_0}}) ds\right] dv_V(x_0),
	\end{equation}
By Lemma \ref{lemmeJq} we get 
\begin{equation}
\begin{aligned}
	\int_{-\infty}^{\infty} J_q(s\sqrt{\alpha_{x_0}}) ds &= \int_{-\infty}^{\infty} \int_{-\infty}^{\infty} \big(\mathrm{erfc}^{q}(y)-\mathrm{erfc}^{q+1}(y)\big)\frac{1}{\sqrt{\pi}}e^{-(y+\sqrt{2\pi \alpha_{x_0}}s)^2}dy ds \\
	&= \frac{1}{\sqrt{2\pi \alpha_{x_0}}} \int_{-\infty}^{\infty} \big(\mathrm{erfc}^{q}(y)-\mathrm{erfc}^{q+1}(y)\big) dy,
\end{aligned}
\end{equation}
to get the last equality we applied Fubini, the change of variable $s' = y + \sqrt{2\pi \alpha_{x_0}}s$, and used the fact that $\int_{-\infty}^{\infty} e^{-s'^2}ds' = \sqrt{\pi}$.
Hence \eqref{expressiona1W} becomes
	\begin{equation}\label{expressiona1finale}
		a_{1,q}(W) =(2\pi)^{-1/2}\rank(E)C(V)C_q,
	\end{equation}
with 
	\begin{align}\label{expressionCq}
	\begin{split}
		 & C_q :=\int_{-\infty}^{\infty} \big(\mathrm{erfc}^{q}(y)-\mathrm{erfc}^{q+1}(y)\big) dy, \\
		& C(V) := \int_{x_0 \in V} \frac{{\det}_{\C} (\mJ_{x_0} \vert_{HV})}{(2\pi)^{n-1}} \sqrt{ \alpha_{x_0}} dv_V(x_0).
	\end{split}
	\end{align}
Notice that if we assume $J_0 = J$, we get  $C(V) =\vol(V)$.

From \eqref{devTraceTWp}, \eqref{expressiona1finale}, \eqref{expressionCq} we obtain for $g:[0, 1]\rightarrow \R$ polynomial vanishing at $0$ and $1$, the asymptotic expansion 
	\begin{equation}\label{trpolyTW}
	\begin{gathered}
		\tr(g(T_{W, p})) = \sum_{k \geq 1} a_{k, g}(W)p^{n-k/2} + \mO(p^{-\infty}), \quad \mathrm{with} \quad a_{1, g} = (2\pi)^{-1/2}\rank(E)C(V) I(g),\\
		\mathrm{and} \quad I(g) = \int_{-\infty}^{\infty} g\left(\mathrm{erfc}(x)\right)dx.
	\end{gathered}
	\end{equation}
In the case where $g$ only vanishes at $0$, applying~\eqref{trpolyTW} to $\tilde{g}(x) := g(x) -g(1)x$ proves Corollary~\ref{ThCEGeneralCase}. \qed

\subsubsection*{\textbf{Proof of Corollary~\ref{CorTraceToeplitz}}} To get Corollary~\ref{CorTraceToeplitz} from Corollary~\ref{ThCEGeneralCase}, it is sufficient to adapt the proof of Charles--Estienne in~\cite[\S 6, pp. 546-548]{CE20}. For~\eqref{casegHolder2}, their proof relies on the fact that the spectrum of $g(T_{W,p})$ is nonnegative and on the Stone--Weierstrass polynomial approximation theorem. To get~\eqref{repartition_spectre}, the idea is to approximate the function $\bbone_{[a, b]}$ from above and below by continuous functions vanishing near $0$ and $1$. \qed

\subsection{Proof of Lemma \ref{lemmeJq}}\label{Appendix}

In this paragraph we prove the technical result of Lemma \ref{lemmeJq}. Recall that

\begin{equation}\label{Jq}
	J_q(s):= \int_{\substack{\Im(z_1), \ldots, \Im(z_q) > -s \\ \Im(z_{q+1}) < -s}}\exp \left(- \frac{\pi}{2} \scal{MZ}{Z}\right) d\lambda(z_1) \ldots d\lambda(z_{q+1}),
\end{equation}
Let $A'$, $B'$ be the $(q+1) \times (q+1)$ matrices defined by 

\begin{equation}
A' := \begin{pmatrix}
		2 & -1 &  & 0\\
		-1 & \ddots & \ddots &   \\
		   & \ddots & \ddots & -1 \\
		0 &	&	-1	&	2	\\
 \end{pmatrix}, \qquad B' := \begin{pmatrix}
		0 & 1 &  & 0\\
		-1 & \ddots & \ddots &   \\
		   & \ddots & \ddots & 1 \\
		0 &	&	-1	&	0	\\
\end{pmatrix}.
\end{equation}
Notice $A'$ is the matrix of the discrete Laplace operator.  For $X = (x_1, x_2, \ldots, x_{q+1})$, $Y = (y_1, y_2, \ldots, y_{q+1})$, $Z = (x_1, y_1, \ldots, x_{q+1}, y_{q+1})$ and since $B'$ is skew-symmetric, we get

\begin{equation}
\scal{MZ}{Z} = \scal{A'X}{X} + \scal{A'Y}{Y} + 2i\scal{X}{B'Y}.
\end{equation}
Hence the equality 

\begin{equation}
\begin{aligned}
	J_q(s) &= \int_{\substack{y_1, \ldots, y_q > -s \\ y_{q+1} <-s}}\int_{(x_1, \ldots, x_{q+1}) \in \R^{q+1}}\exp\left[ -\frac{\pi}{2} (\scal{A'X}{X} + \scal{A'Y}{Y})  \right] \\
	&\mspace{350mu}\exp\left[ i\pi \scal{X}{B'Y}\right] d\lambda(X)d\lambda(Y), \\
	  &= \int_{\substack{y_1, \ldots, y_q > -s \\ y_{q+1} <-s}} \exp\left[ -\frac{\pi}{2} \scal{A'Y}{Y}  \right] F(Y) d\lambda(Y),
\end{aligned}
\end{equation}
and we recognize a Fourier transform
\begin{equation}
\begin{aligned}
	F(Y) &= \int_{(x_1,\ldots ,x_{q+1}) \in \R^{q+1}} \exp\left[-\frac{\pi}{2} \scal{A'X}{X}\right] \exp \left[i\pi \scal{X}{B'Y}\right]d\lambda(X) \\[2ex]
	&= \mathscr{F}\left(\exp\left[-\pi\scal{\frac{1}{2}A'X}{X}\right]\right)\left(-\frac{1}{2}B'Y\right)\\[2ex]
	&= 2^{(q+1)/2}(\det A')^{-1/2}\mathscr{F}\left(\exp\left[-\pi\scal{X}{X}\right]\right)\left(-\frac{1}{\sqrt{2}}(A')^{-1/2}B'Y\right)\\[2ex]
	&= 2^{(q+1)/2}(\det A')^{-1/2} \exp \left[-\frac{\pi}{2} \scal{(A')^{-1/2}B'Y}{(A')^{-1/2}B'Y} \right]\\[2ex]
	&= 2^{(q+1)/2}(\det A')^{-1/2} \exp \left[-\frac{\pi}{2} \scal{B'^T(A')^{-1}B'Y}{Y} \right].
\end{aligned}
\end{equation}
We now have to compute 

\begin{equation}\label{computationJq}
	 J_q(s) = 2^{(q+1)/2}(\det A')^{-1/2}\int_{\substack{y_1, \ldots, y_q > -s \\ y_{q+1} <-s}} \exp\left[ -\frac{\pi}{2} \scal{(A'+ B'^T(A')^{-1}B')Y}{Y}  \right] d\lambda(Y).
\end{equation}
Set 
\begin{equation}
	\begin{gathered}
	E_1 := \mathrm{Diag}(q+1, 2q, 3(q-1),\ldots, 2q, q+1),\\[2ex]
	E_2 := \begin{pmatrix}
			  0  &         &        &  & & \\
			  q  &  0  &        & 0 &  & \\
			  q-1 & 2(q-1) &   0   &   &  & \\
			  q-2 & 2(q-2) & 3(q-2) & \ddots & & \\
			   \vdots   &        &        &  &  &\\
			   1        &  2     & 3 & \cdots & q & 0
		   \end{pmatrix}.
\end{gathered}
\end{equation} 
Notice $B'^T = -B'$, $\det A' = (q+2)$ and the inverse of $A'$ is given by 
\begin{equation}
	A'^{-1} = \frac{1}{q+2}(E_1 + E_2 + E_2^T).
\end{equation}
After computation, we obtain
\begin{equation}
	(q+2) B' A'^{-1}B' = \begin{pmatrix}
							-2q & -(q-2) & 4 & \ldots & 4 \\
							-(q-2) & -2q & -(q-2) & \ddots & \vdots \\
							 4 & -(q-2) & \ddots & \ddots& 4\\
							\vdots & \ddots & \ddots& & -(q-2)\\
							4	& \cdots & 4 &-(q-2) & -2q 
						\end{pmatrix}.
\end{equation}
Therefore we have

\begin{equation}\label{matriceSigmainverse}
	A' + B'^T A'^{-1}B' = 4\Id-\frac{4}{q+2}\mathrm{Ones} = 2\Sigma^{-1},
\end{equation}
with
\begin{equation}\label{definitionSigma}
	\Sigma := \frac{1}{2}(\Id+\mathrm{Ones}), \quad \mathrm{and} \quad \mathrm{Ones}:= \begin{pmatrix}
						1 & \ldots & 1 \\
						\vdots & \ddots & \vdots \\
						1 & \ldots & 1
					\end{pmatrix}.
\end{equation}
Since $\det A' = (q+2)$, from \eqref{computationJq} and \eqref{matriceSigmainverse} we get 

\begin{equation}\label{exprJqs}
\begin{aligned}
	 J_q(s) &= \frac{2^{(q+1)/2}}{\sqrt{q+2}}\int_{\substack{y_1, \ldots, y_q > -s \\ y_{q+1} <-s}} \exp\left( -\pi\scal{\Sigma^{-1} Y}{Y}\right) d\lambda(Y)\\
	 &= \frac{(1/2\pi)^{(q+1)/2}}{\sqrt{\frac{q+2}{2^{q+1}}}}\int_{\substack{y_1, \ldots, y_q > -\sqrt{2\pi}s \\ y_{q+1} <-\sqrt{2\pi}s}} \exp\left( -\frac{1}{2}\scal{\Sigma^{-1} Y}{Y}\right) d\lambda(Y).
\end{aligned}
\end{equation}
In order to compute $J_q$, we will now apply the following result of Steck--Owen \cite{SteckOwen62}.

\begin{thm}[{\cite[(D)]{SteckOwen62}}]
	Let $(X_1, \ldots, X_n)$ be a multivariate equicorrelated Gaussian with $\mathbb{E}(X_i)= 0$, $\mathbb{E}(X_i^2)= 1$ and $\mathrm{Cov}(X_i, X_j) = \rho$ if $i \neq j$. Let us denote by $F_{n, \rho}(s) = \mathbb{P}(X_1 \leq s, \ldots, X_{n} \leq s)$ the distribution function. Then if $\rho > -1/(n-1)$ we have 
	\begin{equation}\label{SteckOwen}
	F_{n, \rho}(s) = \int_{-\infty}^{\infty} \left[G\left(\frac{s-\rho^{1/2}y}{(1-\rho)^{1/2}}\right)\right]^nG'(y)dy,
	\end{equation}
	with $G(x) = \mathrm{erfc}\left(-\frac{1}{\sqrt{2}}x\right)$, $G'(x) = \frac{1}{\sqrt{2\pi}}e^{-\frac{1}{2}x^2}$. The function $G$ is holomorphic and in the case $-\frac{1}{n-1}<\rho < 0$, the expression above still holds when taking $\rho^{1/2} = \sqrt{-1} |\rho|^{1/2}$.
\end{thm}

By the expression \eqref{exprJqs} and the symmetries of the Gaussian law, we can write $J_q(s)$ as 
\begin{equation}
	J_q(s) = \mathbb{P}\left(X_1 \leq \sqrt{2\pi} s, \ldots, X_{q} \leq \sqrt{2\pi}s\right)- \mathbb{P}\left(X_1 \leq \sqrt{2\pi} s, \ldots, X_{q+1} \leq \sqrt{2\pi}s\right),
\end{equation}
where $(X_1, \ldots, X_{q+1})$ are equicorrelated centered Gaussians with covariance matrix $\Sigma$ defined in \eqref{definitionSigma}, therefore $\mathbb{E}(X_i^2) = 1$ and for $i \neq j$, $\mathrm{Cov}(X_i, X_j) = 1/2$. Applying the above theorem we obtain
\begin{equation}\label{expressionJqerfc}
	J_q(s) = \int_{-\infty}^{\infty} \bigg[\left[\mathrm{erfc}\left(-\sqrt{2\pi}s+\frac{1}{\sqrt{2}}y\right)\right]^{q}- \left[\mathrm{erfc}\left(-\sqrt{2\pi}s+\frac{1}{\sqrt{2}}y\right)\right]^{q+1}\bigg]\frac{1}{\sqrt{2\pi}}e^{-y^2/2}dy,
\end{equation}
and by change of variable $y' = -\sqrt{2\pi}s + \frac{1}{\sqrt{2}}y$ we get \eqref{egaliteJq}. This proves Lemma \ref{lemmeJq}. \qed

\bibliography{BiblioToeplitzInd7}{}
\bibliographystyle{amsplain}	
\end{document}